\documentclass[a4paper,11pt]{article}
\usepackage{graphicx}
\usepackage{amsmath}
\usepackage{amssymb}
\usepackage{amstext}
\usepackage{color}
\usepackage{colortbl}
\usepackage{xcolor}

\newcommand{\proof}{\vskip 5truemm \noindent{\textsc Proof.}~}
\newcommand{\finproof}{{\
}\hfill\rule{2mm}{2mm}}
\newtheorem{rem}{Remark}[section]
\newtheorem{thm}{Theorem}[section]
\newtheorem{coro}{Corollary}[section]
\newtheorem{lem}{Lemma}[section]

\newtheorem{defi}{Definition}[section]

\numberwithin{equation}{section}
\begin{document}

\title{\bf BV right continuous solutions of differential inclusions involving  time dependent maximal monotone operators}

\author{Dalila Azzam-Laouir
 \footnote{LAOTI, FSEI, Universit\'e Mohamed Seddik Benyahia de Jijel, Alg\'erie.}
\hskip 4pt Charles Castaing\footnote{  C. Castaing IMAG, Univ
Montpellier, CNRS, Montpellier II, 34095, Case courrier 051,
Montpellier Cedex 5, France. E-Mail: charles.castaing@gmail.com}
\hskip 4pt M. D. P. Monteiro Marques \footnote{CMAF and Faculdade de
Ciencias de Lisboa, Av. Prof. Gama Pinto 2, P. 1600 Lisboa,
Portugal. E-Mail: mmarques@lmc.fc.ul.pt}}

\maketitle
\begin{abstract} This paper is devoted to the study of  evolution problems
involving time dependent  maximal monotone operators, which are
 right continuous and bounded in variation   with respect to the
 Vladimirov's pseudo distance. Several variants and applications are
 presented.

\end{abstract}


\vskip6mm

\section{Introduction}
Let $E$ be a separable Hilbert space,   and $I=[0,T]$ ($T>0$). In
this paper, we are mainly interested by the existence of bounded
variation and right continuous (BVRC) solutions to evolution
inclusion of the form
\begin{equation}\label{1.1} -Du(t) \in A(t) u (t) +F(t,
u(t))\;\;a.e.,\;\;\;u(0)=u_0,\end{equation} governed by a time
dependent maximal monotone operator  $A(t)$, in the vein of
Kunze-Marques work \cite{KM}, with  a weakly compact, convex valued
perturbation $F : I \times E \rightrightarrows E$. For this purpose,
we consider the existence problem of BVRC  solutions to \eqref{1.1}
 by assuming that  $t \mapsto  A(t) $ is of bounded variation
 and right continuous, in the sense  that
 there exists a function  $r:  I \rightarrow [0, \infty[$,
which is right continuous on $[0, T[$  and nondecreasing  with $r(0)
= 0$ and  $r(T) < \infty$ such that
$$ dis(A(t), A(s)) \leq dr(]s, t]) = r(t)-r(s), \hskip 3pt 0\leq s \leq t \leq T, $$
where  $dis(\cdot, \cdot)$ is the pseudo-distance between maximal
monotone operators introduced by Vladimirov \cite{Vla}; see relation
\eqref{2.1}.

When $C(t)$ is a closed convex moving set in $E$, then with $A(t) =
\partial \delta _{C(t)} = N_{C(t)}$, we have
$$dis\big(A(t), A(s)\big) = d_H\big(C(t), C(s)\big), \hskip 2pt t, s \in I,$$ here $d_H$ denotes the Hausdorff
distance. In this regard, our study extends some related results in
the  evolution problem  governed by the  convex sweeping process of
the form \begin{equation}\label{1.2} -Du(t) \in  N_{C(t)} (u (t))
+F(t, u(t))\;\;\;a.e.,\;\;\; u(0)=u_0.\end{equation} There is an
intensive work concerning the existence of solutions  to the
sweeping process. This subject is vast; see \cite{adly-hadd-thib,
M1, Mon, M-SAC}  and the references therein. However, there are  a
few results concerning the existence of BVRC solution to the
sweeping process, see Benabdellah et al \cite{BCSS},
Castaing-Marques \cite{CM}, Adly et al \cite {adly-hadd-thib,ANT},
Edmond-Thibault \cite{ET1}, Nacry et al \cite{NNT}. 

Let $\lambda$ be the Lebesgue measure on $I$ and $dr$  the Stieljes
measure associated with $r$.  We set $\nu := \lambda +dr$ and
$\frac{ d\lambda}{d\nu}$ the density of $ \lambda$ with respect to
the measure $\nu$.  A function $u : I \rightarrow E$ is BVRC if $u$
is of bounded variation and right continuous. By BVRC solution to
\eqref{1.1} we mean that,   giving  $u_0 \in D(A(0))$, there exists
a BVRC  mapping $u: I\to E$ such that
$$
\begin{cases}
u(0)=u_0;\\
 u(t)\in D(A(t))\;\;\;\forall t\in I;\\
   -\displaystyle\frac{du}{d\nu }(t)\in A(t) u(t)+ F(t, u(t))   \frac{d\lambda} {d\nu}(t)  \;\;\;d\nu-a.e.\,t\in
   I,
\end{cases}
$$
here $\frac{du}{d\nu }(t)$ denotes the density of $u$ relatively to
$\nu$. We aim to present the problem of   existence and uniqueness
of BVRC solutions to \eqref{1.1} according to the nature of  the
perturbation $F$, and its applications such as  Skorohod problem,
relaxation, second order evolution,  sweeping process. To the best
of the author's knowledge, this problem of existence and uniqueness
of BVRC solution of \eqref{1.1}  has never been considered in  the
literature before. The study of BV solutions of \eqref{1.1} and
\eqref{1.2} has an increasing interest. Many attempts have been made
to generalize  evolution  inclusion \eqref{1.2}  containing
deterministic or stochastic perturbations \cite{CMPR, FS}.
 Actually, \eqref{1.1} is the first study  of   convex,  weakly compact valued perturbations
 of evolution inclusion governed  by
 time dependent maximal monotone operators  involving the existence of BVRC
 solutions.  For some recent results dealing with existence of absolutely continuous, Lipschitz or continuous with bounded
 variation (BVC) solutions of differential inclusions governed by time dependent or time and state dependent maximal
 monotone operators, we refer to \cite{ABCM1, ABCM2, AB, ACM, BAC, CMRdF1, Ken, Le, SAM,
 TOL1}.

The paper is organized as follows. In section 2, we recall some
preliminary results needed later. In section 3, our main theorems
state the existence of bounded variation and right continuous
solutions to the evolution inclusion \eqref{1.1} when $t \mapsto
A(t) $ is BVRC and the perturbation $F:  I\times E\rightrightarrows
E$ is convex and weakly compact valued,  separately  scalarly upper
semi continuous on $E$ and measurable on $I\times E$, with sharp
application to the existence and uniqueness of BVRC solution to  the
evolution inclusion \begin{equation*}-Du(t) \in A(t) u(t) + f(t,
u(t))\;\;\;a.e.,\;\;\;u(0)=u_0,\end{equation*} where $f : I\times E
\rightarrow E$ is single valued, Borel-measurable, and satisfies a
Lipschitz  condition. Here uniqueness is provided using a specific
result due to Moreau \cite{MV} on the property of BVRC mappings and
some specific Gronwall type lemma. In section 4, we provide some
applications to second order and fractional evolutions, Skorohod
problem and relaxation problem.

The obtained results are quite new with remarkable corollaries in
the setting of BVRC solutions, and they extend to  time dependent
BVRC maximal monotone operators  some nice results of Adly et al
\cite{adly-hadd-thib}, Edmond-Thibault \cite{ET1} and Nacry et al
\cite{NNT},   dealing with the BVRC solutions of convex and proximal
sweeping processes, and a result of Tolstonogov \cite{TOL},  dealing
with the BVC sweeping process.

\section{Notations and  Preliminaries}
In  the whole paper $I:=[0,T]$ $(T>0)$ is an interval of
$\mathbb{R}$ and  $E$ is a separable Hilbert space with the  scalar
product $\langle\cdot,\cdot\rangle$ and the associated norm
$\|\cdot\|$. $\overline{B}_E$ denotes the unit closed ball of $E$
and $r\overline{B}_E$ its closed ball of center 0 and radius $r>0$.
We denote by $\mathcal{L}(I)$ the sigma algebra on $I$,
$\lambda:=dt$ the Lebesgue measure and by $\mathcal{B}(E)$ (resp.
$\mathcal{B}(I))$ the Borel sigma algebra on $E$ (resp. on $I$). If
$\mu$ is a positive measure on $I$, we will denote by $L^p(I, E;
\mu)$ $p \in [1, +\infty[$ (resp. $p=+\infty)$, the Banach space of
classes of measurable functions $u:I\to E$ such that
$t\mapsto\|u(t)\|^p$ is $\mu$-integrable (resp. $u$ is
$\mu$-essentially bounded), equipped with its classical norm
$\|\cdot\|_p$ (resp. $\|\cdot\|_{\infty}$). We denote by
$\mathcal{C}(I,E)$ the Banach space of all continuous mappings $u :
I\to E$, endowed with the sup norm.
\\The excess between closed subsets $C_1 $ and $C_2 $ of $E$, is
defined by $ e(C_1, C_2):=\sup_{x\in C_1}d(x,C_2)$,
 and the   Hausdorff distance between them is given by
$$d_H(C_1, C_2):=\max\Big\{e(C_1, C_2), e(C_2, C_1)\Big\}.$$ The support function of $S\subset E$ is
defined by: $\delta^*(a, S):=\sup_{x\in S}\langle
a,x\rangle$, $\forall a\in E$.\\
If $X$ is a Banach space and $X'$ its topological dual,  we denote
by $\sigma(X, X')$ the weak topology on $X$, and by $\sigma(X', X)$
the weak* topology on $X'$.

Let $A: E \rightrightarrows E$ be a set-valued map. We denote by
$D(A)$, $R(A)$ and $Gr(A)$ its domain, range and graph. We say that
$A$ is monotone, if $ \langle y_1 -y_2, x_1 - x_2\rangle\ge 0$
whenever $x_i \in D(A)$, and $y_i \in A(x_i)$, $i = 1, 2$.
 In addition, we say that $A$ is a maximal monotone operator of $E$, if
its graph could not be contained properly in the graph of any other
monotone operator. By Minty's Theorem, $A$ is maximal monotone iff
$R(I_E+A)=E$, where  $I_E$ is the identity mapping of $E$.
\\
If $A$ is a maximal monotone operator of $E$, then  for every $x \in
D(A)$, $A(x)$ is nonempty closed and convex. We denote the
projection of the origin on the set $A(x)$ by $A^0(x)$.

 Let $\eta > 0$, then the resolvent and the Yosida approximation
 of $A$ are the
 well-known operators defined respectively by  $J_{\eta}^A= (I_E +
\eta A)^{-1}$ and $A_{\eta} = \frac{1}{ \eta} (I_E- J_{\eta}^A)$.
  These operators
are single-valued and defined on all of $E$, and we have
$J^A_{\eta}(x)\in D(A)$, for all $x\in E$. For more details about
the theory of maximal monotone operators we refer the reader to
\cite{Ba, brezis, Vra}.

 Let $A: D(A)\subset E\to 2^E$ and $B: D(B)\subset
E\to 2^E$ be two maximal monotone operators, then we denote by
$dis(A, B)$ the pseudo-distance between $A$ and $B$ defined by
\begin{equation}\label{2.1} dis(A,B)=\sup\bigg\{\frac{\langle y-y', x'-x\rangle}{1+\|
y\|+\| y'\|}:\;x\in D(A),\;y\in Ax,\;x'\in D(B),\;y'\in
Bx'\bigg\}.\end{equation}
Our main results are established under the following hypotheses:\\
$(H_1)$ there exists a function $r:I\to [0, +\infty[$ which is right
continuous on $[0, T[$ and nondecreasing with $r(0) = 0$  and
$r(T)<+\infty$ such that \begin{equation*}dis(A(t), A(s))\leq dr(]s,
t])=r(t)-r(s)\;\;\textmd{for}\;\;0\leq s\leq t\leq T.\end{equation*}
$(H_2)$ There exists a nonnegative real constant $c$ such that
\begin{equation*}\|A^0(t,x)\|\leq c(1+\| x\|)\;\;\textmd{for}\;\;t\in I,\;x\in
D(A(t)).\end{equation*} $(H_3)$ $\underset{t\in I}{\bigcup} D(A(t))$
is ball compact, i.e, its intersection with any closed ball of $E$
is compact. \vskip2mm

For the proof of our main theorems we will need some elementary
lemmas taken from reference \cite{KM}.
\begin{lem}\label{lem2.1}
Let $A$ be a maximal monotone operator of $E$. If $x\in
\overline{D(A))}$ and  $y\in E$ are such that
$$ \langle A^0(z)-y, z-x\rangle\geq 0\;\;\forall z\in D(A),$$
then $x\in D(A)$ and $y\in A(x)$.
\end{lem}
\begin{lem}\label{lem2.2}
Let $A_n$ $(n\in \mathbb{N})$ and $A$ be maximal monotone operators
of $E$ such that $dis(A_n, A)\to 0$. Suppose also that $x_n\in
D(A_n)$ with $x_n\to x$ and $y_n\in A_n(x_n)$ with $y_n\to y$ weakly
for some $x, y\in E$. Then $x\in D(A)$ and $y\in A(x)$.
\end{lem}
\begin{lem}\label{lem2.3}
Let $A$ and $B$ be maximal monotone operators of $E$. Then\\
1) for $\eta>0$ and $x\in D(A)$
$$ \|x-J_{\eta}^{B}(x)\| \leq \eta\|A^0(x)\|+dis(A,B)+\sqrt{\eta\big(1+\|A^0(x)\|\big)dis(A, B)}.$$
2) For $\eta>0$ and $x, x'\in E$
$$ \|J_{\eta}^{A}(x)-J_{\eta}^{A}(x')\|\leq
\|x-x'\|.$$
\end{lem}
\begin{lem}\label{lem2.4}
Let $A_n$ $(n\in \mathbb{N})$ and $A$ be maximal monotone operators
of $E$ such that $dis(A_n, A)\to 0$ and $\|A^0_n(x)\|\leq
c(1+\|x\|)$ for some $c>0$, all $n\in \mathbb{N}$ and $x\in D(A_n)$.
Then for every $z\in D(A)$ there exists a sequence $(\zeta_n)$ such
that
 \begin{equation}\label{2.2}\zeta_n\in
D(A_n),\;\;\;\zeta_n\to z\;\;\textmd{and}\;\;A_n^0(\zeta_n)\to
A^0(z).\end{equation}
\end{lem}

We finish this section by some  types of Gronwall's lemma, which are
crucial for our purpose.
\begin{lem}\label{lem2.5}
Let $(\alpha_i)$, $(\beta_i)$, $(\gamma_i)$ and $(a_i)$ be sequences
of nonnegative real numbers such that $a_{i+1}\leq
\alpha_i+\beta_i\big(a_0+a_1+....+a_{i-1}\big)+(1+\gamma_i)a_i$ for
$i\in \mathbb{N}$. Then
$$ a_j\leq \bigg(a_0+\sum_{k=0}^{j-1}
\alpha_k\bigg)\exp\bigg(\sum_{k=0}^{j-1}
\big(k\beta_k+\gamma_k\big)\bigg)\;\;\textmd{for}\;j\in
\mathbb{N}^*.$$
\end{lem}

\begin{lem}\label{lem2.6}  Let $\mu$ be a positive Radon measure on $I$. Let $g \in L^1(I, \mathbb{R}; \mu)$ be a nonnegative function
 and $\beta  \geq 0$ be such that,
$\forall t \in I$, $0\leq  \mu ( \{t\} ) g(t)  \leq \beta < 1$. Let
$\varphi \in  L^\infty(I, \mathbb{R}; \mu)$ be a nonnegative
function satisfying
\begin{equation*}\varphi(t)  \leq \alpha +\int_{]0, t]}  g(s) \varphi(s)
\mu(ds)\;\;\;  \forall t \in I, \end{equation*} where $\alpha$ is a
nonnegative constant. Then
\begin{equation*}\varphi(t) \leq \alpha \exp \Big( \frac{1} {1-\beta} \int_{]0, t]} g(s) \mu(ds)\Big)\;\;\;    \forall t \in I. \end{equation*}
\end{lem}

\proof  This lemma is due to M.M. Marques. For a proof, see e.g
(\cite{BCG}, Lemma 2.1). \finproof

\begin{lem} \label{lem2.7}  Let $\mu$ be a non-atomic positive Radon measure on the interval $I$. Let $c$, $p$ be nonnegative real functions
such that $c \in L^1(I, \mathbb{R};\mu), p \in L^\infty (I,
\mathbb{R};\mu)$, and let $\alpha\geq 0$. Assume that for
$\mu-a.e.\, \, t\in I$
$$p(t) \leq \alpha +\int_{0}^{t} c(s) p(s) \mu(ds).$$
Then, for
$\mu-a.e.\, \, t\in I$
$$p(t) \leq \alpha \exp\Big(\int _{0}^{t} c(s)
\mu(ds)\Big).$$
\end{lem}
The proof (see \cite{ACM}, Lemma 2.7; or \cite{M-SAC}, Lemma 4,
taking $\eta =0$) is not a consequence of the classical Gronwall
lemma dealing with Lebesgue measure $\lambda$ on $I$. It relies on a
deep result of Moreau-Valadier on the derivation of (vector)
functions of bounded variation \cite {MV}.

\begin{lem} \label{lem2.8} (Proposition 4.1 in \cite{TOL}) Let $m\in L^1(I, \mathbb{R};
\lambda)$ be a nonnegative function, and let $x:I\longrightarrow [0,
+\infty[$ be a right continuous function of bounded variation. If
$$ \frac{1}{2} x^2(t)\leq \frac{1}{2} a^2+\int_{]0, t]} m(s)
x(s)ds,\;\;\;t\in I,\;\;\;a\geq 0,$$ then
$$  x(t)\leq  a+\int_{]0, t]} m(s)
ds,\;\;\;t\in I.$$
\end{lem}
\section{Main results: Existence and uniqueness of BVRC solutions}

We recall, unless stated, that in all the paper,   $E$ is a
separable real Hilbert space, $\lambda$ is the Lebesgue measure on
$I$, $dr$ is the
 Stieljes  measure associated with $r$,   $\nu := \lambda +dr$  and $\frac{d\lambda}{d\nu}$
 is the density of $\lambda$ with respect to the measure $\nu$.

In this section we are interested by the existence of bounded
variation right continuous (shortly  BVRC)  solutions   to the
inclusion  \eqref{1.1}. For the sake of completeness let us   state
and summarize  some useful facts. We refer to  \cite{BCG, Por} for
the proof.
\begin{thm}\label{Theorem 3.1}  Let $X$ be a separable Banach space,    $(I, {\mathcal T}_{\mu}, \mu)$
be a measure space, where   $\mu$ is a positive Radon measure, and
let    $\Gamma:I\rightrightarrows X$ be a convex weakly compact
valued multi-mapping, which is scalarly ${\mathcal
T}_{\mu}$-measurable and such that $\Gamma(t) \subset m(t)
{\overline B}_X$ for some nonnegative function $m\in L^1(I,
\mathbb{R}; \mu)$.  Let $S^1_\Gamma$ be the set of all $L^1(I, X;
\mu)$-selections of $\Gamma$, i.e,
$$ S^1_\Gamma=\Big\{\phi\in L^1(I, X;
\mu):\;\phi(t)\in \Gamma(t)\;\forall t\in I\Big\}.$$ Then the
following properties hold.
\\
(i)  $S^1_\Gamma$ is convex weakly compact in  $L^1(I, X; \mu)$.
\\
(ii)  Let $x_0\in X$.  For each $h \in  S^1_\Gamma$, the mapping $t
\in I\mapsto u_h (t):= x_0 + \int_{]0, t]} h(s) d\mu(s)$ is BVRC
with $\frac{du_h} {d\mu} = h$ $\mu$-a.e.,  and the set $\{ u_h :\; h
\in S^1_\Gamma \}$ is equi-right continuous with bounded variation,
i.e., for  all $h \in S^1_\Gamma$
$$\|u_h(t) -u_h(\tau) \| \leq \int_{]\tau, t]} m (s)
d\mu(s)\;\;\;\textmd{for all}\;\;0\leq \tau \leq t \leq T.$$
 (iii) Let $(h_n)$ be a sequence in  $S^1_\Gamma$, then by extracting
 a
subsequence,  that we do not relabel, $(h_n)$ converges weakly to
some mapping $h\in S^1_\Gamma$, so that $(u_{h_n})$ pointwise
converges weakly to  the BVRC mapping $u_h$, with for all $t\in I$,
$ u_h (t) = x_0 + \int_{]0, t] } h(s)  d\mu(s)$ and  $\frac{du_h}
{d\mu} = h$ $\mu$-a.e.
\\
(iv) Assume further that for each $ t  \in I$,   $\{ u_{h_n} (t) :\;
n\in \mathbb{N}\}$ is relatively compact, then  $(u_{h_n})$
pointwise converges strongly to $u_h$.
\\
(v) Assume that $\Gamma(t)$ is convex compact for each $ t  \in I$,
then   the set $\{ u_{h_n} (t) : \;n\in \mathbb{N}\}$ is relatively
compact and $(u_{h_n})$ pointwise converges strongly to $u_h$.
\end{thm}

Now, we proceed to state  the main existence  results.
We begin with an existence of  a {\it second order BVRC  solution}  to our  evolution inclusion.
\begin{thm}\label{Theorem 3.2}   Let $f : I \longrightarrow E$
be a $\lambda$-measurable mapping  such that $\|f(t)\| \leq M$, for
all $t \in I$, for some nonnegative real constant $M$. Let for every
$t\in I$, $A(t):D(A(t))\subset E\rightrightarrows E$ be a maximal
monotone operator satisfying $(H_1)$, $(H_2)$ and $(H_3)$.
\\Then for any $x_0\in E$,  $u_0\in D(A(0))$, there exists a unique
BVRC solution $(x, u) :I \to E\times E$   to the problem
$$
\begin{cases}
x(t)=x_0+\displaystyle\int_{]0, t] }  u(s)d\nu (s)\;\;\;\forall t\in I;\\
u(0)=u_0;\\
 u(t)\in D(A(t))\;\;\;\forall t\in I;\\
   -\displaystyle\frac{du}{d\nu }(t)\in A(t) u(t)+ f(t)  \frac{d\lambda} {d\nu}(t)  \;\;\;\nu-a.e.\,t\in
   I.
\end{cases}
$$
with the estimate $\frac{du}{d\nu }(t) \in K\overline B_E$
$\nu$-a.e., where $K$ is a positive constant.

\end{thm}

\proof We choose a sequence $(\varepsilon_n)_n\subset ]0, 1]$, which
decreases to $0$ as $n\to \infty$ and partition
$0=t_0^n<t_1^n<...<t_{k_n}^n=T$ of $I$ such that
\begin{equation}\label{3.1}
|t_{i+1}^n-t_i^n|+dr(]t_i^n, t_{i+1}^n])\leq
\varepsilon_n\;\;\;\textmd{for}\;i=0,...,k_n-1.\end{equation} We set
$I_0^n=\{t_0^n\}$ and  $I_i^n=]t_i^n, t_{i+1}^n]$ for
$i=0,...,k_n-1$.\\Such a partition  can  be obtained by considering
the measure $\nu = dr +\lambda $ using the constructions developed
in Castaing et al \cite{CM}. For $i = 0,...,k_n-1$, let
\begin{equation}\label{3.2} \delta_{i+1}^n=dr(]t_i^n,
t_{i+1}^n])=r(t_{i+1})-r(t_i^n), \;\;\;\; \eta^n_{i+1} = t
^n_{i+1}-t ^n_i, \;\;\;\;  \beta ^n_{i+1}  = \nu(]t^n_i, t^n
_{i+1}]).\end{equation} Let us define, for every $n\geq 1$,
sequences $(x_i^n)_{0\leq i\leq k_n-1}$ and $(u_i^n)_{0\leq i\leq
k_n-1}$ such that $x_0^n=x_0$, $u_0^n=u_0\in D(A(0))$, and for $i =
0,...,k_n-1$, \begin{equation}\label{3.3} u_{i+1}^n= J^n_{i+1}
\Big(u_i^n-\int_{t_i^n}^{t_{i+1}^n}
f(s)d\lambda(s)\Big)\end{equation} and
\begin{equation}\label{3.4}
x_{i+1}^n=x_i^n+\beta_{i+1}^nu_{i+1}^n, \end{equation} with
$J^n_{i+1}:=J^{A(t_{i+1}^n)}_{\beta^n_{i+1}} =\big(I_E+
\beta^n_{i+1} A(t_{i+1}^n)\big)^{-1}$. \\Remark that by the
definition of the resolvent  we have  $u^n_{i+1}  \in
D(A(t_{i+1}^n))$ and
\begin{equation}\label{3.5} -\frac{1} {\beta^n_{i+1}}\Big(u^n_{i+1}-
u_i^n+\int_{t_i^n}^{t_{i+1}^n} f(s)d\lambda(s)\Big) \in A(
t_{i+1}^n) u^n _{i+1}. \end{equation} For $t\in [t^n_i, t
^n_{i+1}[$,  $i = 0,...,k_n-1$, set
\begin{equation}\label{3.6}
x_n(t)=x_{i}^n+ \frac{ \nu(]t^n_i, t])} {  \nu(]t^n_i, t^n
_{i+1}])}( x_{i+1}^n-x_i^n)\end{equation} and
\begin{equation}\label{3.7}
v_n(t)=   u_i^n + \frac{ \nu(]t^n_i, t])  }{ \nu(]t^n_i, t
^n_{i+1}])} \Big(u_{i+1}^n-u_i^n+\int_{t_i^n}^{t_{i+1}^n} f(s)
d\lambda(s)\Big)  - \int_{t_i^n}^{t} f(s)d\lambda(s),\end{equation}
so that $v_n$, $x_n$   are of bounded variation and right continuous
on $I$, with $v_n(t_i^n)=u_i^n$ and $x_n(t_i^n)=x_i^n$.
\\

 \textit{Step 1.}
 Let us show that the sequence $(v_n)$ of step approximations
is uniformly bounded in  norm and variation.\\We have from
\eqref{3.3}, Lemma \ref{lem2.3}, $(H1)$,  $(H2)$ and the boundedness
of $f$, for $i=0,...,k_n-1$,
\begin{eqnarray*}
\|
u_{i+1}^n-u_i^n\|&=&\Big\|J_{i+1}^n\Big(u_i^n-\int_{t_i^n}^{t_{i+1}^n}
f(s)d\lambda(s)\Big)-u_i^n\Big\|\\
&\leq& \Big\|J_{i+1}^n\Big(u_i^n- \int_{t_i^n}^{t_{i+1}^n}
f(s)d\lambda(s)\Big)-J_{i+1}^n(u_i^n)\Big\|+\big\|J_{i+1}^n(u_i^n)-u_i^n\big\|\\
&\leq& \int_{t_i^n}^{t_{i+1}^n} \|f(s)\|
d\lambda(s)+\beta_{i+1}^n\|A^0(t_i^n,
u_i^n)\|+dis\big(A(t_{i+1}^n),A(t_i^n)\big)\\&+&\sqrt{\beta_{i+1}^n\big(1+\|A^0(t_i^n,
u_i^n)\|\big)dis\big(A(t_{i+1}^n),A(t_i^n)\big)}\\
&\leq& M\beta_{i+1}^n +\big(1+c(1+\|u_i^n\|)\big)\beta_{i+1}^n+
\sqrt{\big(1+c(1+\|u_i^n\|)\big)(\beta_{i+1}^n)^2}\\
&\leq& M\beta_{i+1}^n +\big(1+c(1+\|u_i^n\|)\big)\beta_{i+1}^n
+\big(1+c(1+\|u_i^n\|)\big)\beta_{i+1}^n,
\end{eqnarray*}
that is, \begin{equation}\label{3.8}\| u_{i+1}^n-u_i^n\|\leq \bigg(
2c\|u_i^n\|+2(1+c)+M\bigg)\beta_{i+1}^n.\end{equation} Then,
$$ \|u_{i+1}^n\|\leq
\big(1+2c\beta_{i+1}^n\big)\|u_i^n\|+\big(2(1+c)+M\big)\
\beta_{i+1}^n.$$ By Lemma \ref{lem2.5} we get
\begin{eqnarray*}\|u_i^n\|&\leq&
\Big(\|u_0\|+\big(2(1+c)+M\big)\sum_{k=0}^{i-1}
\beta_{k+1}^n\Big)\exp\Big((2c)\sum_{k=0}^{i-1}
\beta_{k+1}^n\Big)\\
&\leq&\Big(\|u_0\|+\big(2(1+c)+M\big)\nu(]0,T])\Big)\exp\Big(2c
\nu(]0,T])\Big)=:K_1,
\end{eqnarray*}
and by \eqref{3.8}
$$\|u_{i+1}^n-u_i^n\|\leq\Big(2c K_1+2(1+c)+M\Big)\beta_{i+1}^n=:K_2\beta_{i+1}^n.$$
  So that, if we set $K=max\{K_1,K_2\}$, we conclude that for $0\leq i \leq k_n$, resp. $i<k_n$:
\begin{equation}\label{3.9}\|u_i^n\|\leq K,\;\;
\textmd{resp.} \;\;\|u_{i+1}^n-u_i^n\|\leq K\nu(]t_i^n, t_{i+1}^n]).
\end{equation}
Now, for $t\in [t_i^n, t_{i+1}^n[$, we have from \eqref{3.1},
\eqref{3.7} and  \eqref{3.9},
\begin{eqnarray}\label{3.10}
\|v_n(t)-u_i^n\| &=&  \Big\| \frac{ \nu(]t^n_i, t])  }{ \nu(]t^n_i,
t ^n_{i+1}])} \Big(u_{i+1}^n-u_i^n+\int_{t_i^n}^{t_{i+1}^n} f(s)
d\lambda(s)\Big)  - \int_{t_i^n}^{t}
f(s)d\lambda(s)\Big\|\nonumber\\
&\leq& \|u_{i+1}^n-u_i^n\|+2M\delta_{i+1}^n \leq (K+2M)
\varepsilon_n =: M_1\varepsilon_n.\end{eqnarray}
  Furthermore, from \eqref{3.9} and \eqref{3.10}, for all $n\in
  \mathbb{N}$,
\begin{equation*}\|v_n(t)\|\leq K+M_1\varepsilon_n
\leq K+M_1=:M_2\;\;\;\forall t\in I,\end{equation*}  that is
\begin{equation}\label{3.11}
\sup_n\|v_n\|=\sup_n\big(\sup_{t\in I}\|v_n(t)\|\big)\leq M_2.
\end{equation}
On the other hand, if we fix
 $s\in [t_i^n, t_{i+1}^n[$ and $t\in [t_j^n, t_{j+1}^n[$ with
$j>i$, we get by \eqref{3.9} and \eqref{3.10},
\begin{eqnarray*}
\|v_n(t)-v_n(s)\|&\leq&
\|v_n(t)-u_j^n\|+\|u_j^n-u_i^n\|+\|v_n(s)-u_i^n\|\\ &\leq&
2M_1\varepsilon_n+ \sum_{k=0}^{j-i-1}
\|u_{i+k+1}^n-u_{i+k}^n\| \\
&\leq&2M_1\varepsilon_n+ K \sum_{k=0}^{j-i-1} \nu\big(]t_{i+k}^n,
t_{i+k+1}^n]\big)=2M_1\varepsilon_n+K\nu(]t_i^n, t_j^n])\\&\leq&
2M_1\varepsilon_n+ K\nu(]t_i^n, t]) \leq2M_1\varepsilon_n+
K\big(\nu(]t_i^n, s])+\nu(]s, t])\big)\\&\leq&2M_1\varepsilon_n+
K\big(\nu(]t_i^n, t_{i+1}^n[)+\nu(]s, t])\big).
\end{eqnarray*}
Finally, from \eqref{3.1}, we obtain for $n\in \mathbb{N}$ and
$0\leq s\leq t\leq T$,
\begin{equation}\label{3.12} \|v_n(t)-v_n(s)\|\leq
K\nu(]s,t])+(K+2M_1) \varepsilon_n.\end{equation}

\textit{ Step 2.} Convergence of the sequences $(v_n)$ and
$(x_n)$.\\
Let us  define $\theta_n: I\to I$ through
$$\theta_n(t)=t_{i+1}^n\,\quad {\rm for}\,\, t\in ]t_i^n,
t_{i+1}^n],\;i=0,1,...,k_n-1,$$ and $\theta_n(0)=0$. By \eqref{3.3},
we know that for all $t\in I$, $v_n(\theta_n(t))\in
D\big(A(\theta_n(t))\big)$, using hypothesis $(H_3)$, we
conclude that $\big(v_n(\theta_n(t))\big)$ is relatively compact.\\
On the other hand, observe that for all $t\in I$, $\nu(]t,
\theta_n(t)])\to 0$ as $n\to\infty$. So that, from \eqref{3.12},
 $\|v_n(\theta_n(t))-v_n(t)\|\to 0$. That is for all $t\in I$,
 $(v_n(t))$ is also  relatively compact.

Since the sequence $(v_n)$ of BVRC mappings is uniformly bounded in
variation and in norm such that   $(v_n(t)), t \in I$,  is  {\bf relatively compact}, using Helly-Banach's theorem \cite{Por},  
we may assume that there is a BV mapping $u:I\longrightarrow E$ such
that $(v_n(t))$ converges strongly to $u(t)$ for every $t\in I$.
 In particular, $u(0)=u_0$ and by taking the limit in \eqref{3.12}, we get
$$ \|u(t)-u(s)\|\leq K \nu(]s,t])\;\;\;\; \textmd{for}\; 0\leq s\leq t\leq T.$$
It is clear that $u$  is BVRC, and that, $\|du\|\leq Kd\nu$ in the
sense of the ordering of real measures and there exists a density
$u'$ of $du$ with respect to $d\nu$: $du= u' \, d\nu$.
 Consequently, $\| u'\|_{\infty}\leq K$, so that $ u'\in L^1(I, E;
 d\nu)$.\\
 Moreover, since
 $\|v_n(t)-u(t)\|\to 0$, we obtain for all $t\in I$,
 \begin{equation}\label{3.15}\|v_n(\theta_n(t))-u(t)\|\leq
\|v_n(t)-u(t)\|+\|v_n(t)-v_n(\theta_n(t))\|\to
0\;\;\;\textmd{as}\;n\to\infty.\end{equation}

Next remark that  $dv_n= v'_n\, d\nu$, where the density $v'_n$ is
given $d\nu$-almost everywhere by
\begin{equation}\label{3.13}
v'_n(t)=\frac{1}{\beta_{i+1}^n}\Big(u_{i+1}^n-u_i^n+\int_{t_i^n}^{t_{i+1}^n}
f(s)d\lambda(s)\Big)-f(t) \frac{d\lambda}{d\nu} (t)
\;\;\textmd{for}\;t\in ]t_i^n, t_{i+1}^n[.\end{equation} So that, by
\eqref{3.9} and the boundedness of $f$, we get
\begin{equation}\label{3.14} \| v'_n\|_{\infty}\leq
K+2M.\end{equation}  Extracting   a subsequence (not relabeled) we
may assume  that $( v'_n)$ converges weakly in $L^1(I, H; d\nu)$ to
some mapping $w\in L^1(I, H; d\nu)$ with $\|w(t\| \leq K+2M$
$d\nu$-a.e. Whence for $t\in ]0, T]$
\begin{eqnarray*}
 u(t)-u(0) &=&\lim_{n\to\infty}  \big(v_n(t)-v_n(0)\big)=\lim_{n\to\infty}
dv_n(]0,t])\\
&=& \lim_{n\to\infty}  \int_{]0,t]}  v'_n d\nu=\lim_{n\to\infty}
\int \textbf{1}_{]0,t]}
 \ v'_n d\nu
=\int \textbf{1}_{]0,t]} w d\nu
\end{eqnarray*}
that is
$$ ( u' d\nu)(]0,t])=du(]0, t])=u(t)-u(0)=\int_{]0,t]} w d\nu=(w
d\nu)(]0,t])\;\;\forall t\in I.$$ Thus $ u'=w$, $d\nu$-a.e. in $I$,
i.e.,  $(v'_n)$ converges weakly to $ u'$ in $L^1(I, H; d\nu)$.

  Now we note that
$$ x_n(t) = x_0 + \int_{ ]0, t] } v( \theta_n(s)) d\nu(s) \;\;\;\;
\forall t \in I.  $$ Indeed,
\begin{eqnarray*}&&x_0 + \int_{ ]0, t] } v( \theta_n(s)) d\nu(s)\\ &=&x_0 +
\int_{ ]0, t ^n_1] } v( \theta_n(s)) d\nu(s)+   \int_{ ]t^n_1, t
^n_2] } v( \theta_n(s)) d\nu(s) + \cdots  +  \int_{ ]t^n_i, t] } v(
\theta_n(s)) d\nu(s)\\
 &=& x_0+ \beta^n_1 u^n_1+  \beta^n_2 u^n_2+ \cdots+ \nu(]t^n_i, t])
u^n_{i+1}\\
 &= &x ^n_1+  \beta^n_2 u^n_2+ \cdots + \nu(]t^n_i, t]) u^n_{i+1}\\
&=&x^n_i + \nu(]t^n_i, t])  \frac{ x^n_{i+1}-x^n_i}{\beta^n_{i+1}} =
x_n(t). \end{eqnarray*} As consequence we obtain from \eqref{3.9}
and \eqref{3.15},
\begin{equation*}\lim_{n \to \infty} x_n(t) = x_0+ \lim_{n \to \infty}
\int_{]0, t ] } v_n ( \theta_n(s)) d\nu(s) = x_0+  \int_{]0, t ] }
u(s)  d\nu(s) =: x(t),\end{equation*} that is $dx = u$ $\nu$-a.e.
\\

\textit{Step 3.} We are going to show in this step that  
$- u' (t)
-f(t)\frac {d\lambda}{d\nu }(t) \in A(t) u(t)) $ $\nu$-a.e.\\
 Referring to  \eqref{3.5} and \eqref{3.13}, there
is a $\nu$-null set $N_n$ such that
\begin{equation}\label{3.16}-\frac{dv_n} {d\nu} (t) - f(t)
\frac{d\lambda}{d\nu} (t) \in A (\theta_n(t))
v_n(\theta(t))\;\;\;\forall t\in I\setminus N_n,
\end{equation}
further \begin{equation}\label{3.17} v_n(\theta_n(t))\in D(A
(\theta_n(t)))\;\;\;\forall t\in I.
\end{equation}
So,  by \eqref{3.15}, \eqref{3.17} and the fact that
$dis(A(\theta_n(t)), \theta(t))\to 0$ as $n\to\infty$, using Lemma
\ref{lem2.2}, we conclude that $u(t) \in D(A(t))$ for all $t\in I$.
Consequently, for our goal, using Lemma \ref{lem2.1}, it is enough
to check  that for
  $\nu $ almost every $t\in I$
 and for all $\gamma\in D(A(t))$,
\begin{equation*}
\Big\langle  \frac{du}{d\nu}(t)+  f(t) \frac{d\lambda}{d\nu} (t) ,
u(t)-\gamma\Big\rangle\leq \Big\langle A^0(t,\gamma),
\gamma-u(t)\Big\rangle.
\end{equation*}
Indeed, since for all $t\in I$, $dis(A(\theta_n(t)),A(t))\to 0$ as
$n\to\infty$ and since $(H_2)$ is satisfied, we may apply Lemma
\ref{lem2.4}, to find a sequence $(\zeta_n)$ such that
\begin{equation}\label{3.18} \zeta_n\in
D\big(A(\theta_n(t))\big),\;\; \zeta_n\to \gamma\;\;
\textmd{and}\;\; A^0(\theta_n(t), \zeta_n)\to
A^0(t,\gamma).\end{equation} Since $A(t)$ is monotone, in particular
by \eqref{3.16}, for $t\in I\setminus N_n$
\begin{equation}\label{3.19}\Big \langle \frac{dv_n} {d\nu}(t)+f(t) \frac
{d\lambda}{d\nu }(t), v_n(\theta_n(t))-\zeta_n\Big\rangle\leq
\Big\langle A^0(\theta_n(t),\zeta_n),
\zeta_n-v_n(\theta_n(t))\Big\rangle.\end{equation} Since $( v'_n)$
weakly converges to $u'$ in $L ^1(I, E; d\nu)$, $(v'_n)$ Komlos
converges $\nu$-a.e to $u'$, then there is a negligible set $N$ such
that for $t \notin N$ \begin{equation}\label{3.20}\lim_{n\to\infty}
\frac {1}{n} \sum_{j= 1} ^n \ v'_j (t) =
u'(t)=\frac{du}{d\nu}(t).\end{equation}
 Whence, since
 \begin{eqnarray}\label{3.21}  &&\big\langle  v'_n(t) +f(t)\frac {d\lambda}{d\nu }(t), u(t) - \gamma  \big\rangle =
 \big\langle  v'_n(t)+f(t ) \frac {d\lambda}{d\nu }(t) , v_n(\theta_n(t))-\zeta_n
 \big\rangle\nonumber
\\&+& \big\langle  v'_n(t)+ f(t)\frac {d\lambda}{d\nu }(t) ,  u(t)
-v_n(\theta_n(t)) \big\rangle+ \big\langle  v'_n(t)+ f(t)\frac
{d\lambda}{d\nu }(t),  \zeta_n -\gamma \big\rangle,\end{eqnarray}
then
\begin{eqnarray*}&&\frac{1}{n} \sum_{j= 1}^n \big\langle  v'_j
(t)+ f(t)\frac {d\lambda}{d\nu }(t),  u(t) - \gamma \big\rangle=
 \frac{1}{n}  \sum_{j= 1}^n \big\langle  v'_j(t)+ f(t)\frac {d\lambda}{d\nu }(t)  , v_j(\theta_j(t))-\zeta_j
 \big\rangle\\
&+& \frac{1}{n}  \sum_{j= 1}^n   \big\langle  v'_j(t)+ f(t)\frac
{d\lambda}{d\nu }(t) ,  u(t) -v_j(\theta_j(t)) \big\rangle +
\frac{1}{n}  \sum_{j= 1}^n \big\langle  v'_j(t)+  f(t)\frac
{d\lambda}{d\nu }(t) ,  \zeta_j  -\gamma \big\rangle,
\end{eqnarray*}
so that, by \eqref{3.19}, \eqref{3.14} and since $f$ is bounded, for
$t\in I\setminus (\underset{n}{\cup}N_n\cup N)$,
\begin{eqnarray*}&&\frac{1}{n} \sum_{j= 1}^n \big\langle v'_j (t)+
f(t)\frac {d\lambda}{d\nu }(t), u(t) - \gamma \big\rangle \leq
\frac{1}{n}  \sum_{j= 1}^n \big\langle  A^0(
\theta_j(t), \zeta_j) , \zeta_j   -v_j(\theta_j(t)) \big\rangle\\
&+& (K+ 3M) \frac{1}{n}  \sum_{j= 1}^n     \|u(t) -v_j(\theta_j
(t))\| + (K+ 3M)  \frac{1}{n} \sum_{j= 1}^n \|\zeta_j
-\gamma\|.\end{eqnarray*} Passing to the limit  when $n\rightarrow
\infty$, in this inequality, we get by  \eqref{3.20}, \eqref{3.18}
and \eqref{3.15} \begin{equation*}\big\langle  u'(t)+ f(t)\frac
{d\lambda}{d\nu }(t), u(t) -\gamma \big\rangle \leq  \big\langle
A^0(t, \gamma), \gamma-u(t) \big\rangle\;\;\;\nu-a.e.\end{equation*}
  This inequality can be also obtained by
applying the Mazur's trick to $(v'_n)$ via the inequality
\eqref{3.21}.
 As a consequence,
$-\frac{du}{d\nu}(t) -f(t)\frac {d\lambda}{d\nu }(t)  \in A(t) u(t))
$ $\nu$-a.e. with $u(0)=u_0$, and
 $\| u'(t)\|\leq K, \hskip 2pt  \nu-a.e.\, t\in I$. \\The uniqueness
 of the solution is a consequence of the monotonicity of $A(t)$.
This completes our proof. \finproof
\\

As a by product of Theorem \ref{Theorem 3.2} we mention  some useful
application.

\begin{thm} \label{Theorem 3.3}
Under the hypotheses of Theorem \ref{Theorem 3.2}, for any  $u_0\in
D(A(0))$, there exists a unique BVRC solution $u :I \to E$ to the
problem
$$
\begin{cases}
u(0)=u_0;\\
 u(t)\in D(A(t))\;\;\;\forall t\in I;\\
 \frac{du}{d\nu}\in L^{\infty}(I, E; d\nu)\\
   -\displaystyle\frac{du}{d\nu }(t)\in A(t) u(t)+ f(t)  \frac{d\lambda} {d\nu}(t)  \;\;\;\nu-a.e.\,t\in
   I.
\end{cases}
$$
\end{thm}

In the following we apply our result and tools developed above to
the sweeping process.

\begin{thm} \label{Theorem  3.4}
For every $t\in I$,  let us consider  the  closed convex valued
mapping  $C: I \rightrightarrows E$  such that $d_H (C(t), C(s))
\leq r(t) - r(\tau), \forall  \tau \leq t \in I$ and C(t) is ball-compact. Let $f : I \to E$
be a bounded $\lambda$-measurable mapping.    Then for all $u_0\in
C(0)$, there is a unique  BVRC solution $u(\cdot)$ to the evolution
problem
$$
\begin{cases}
u(0)=u_0;\\
 u(t)\in C(t)\;\;\;\forall t\in I;\\
 \frac{du}{d\nu}\in L^{\infty}(I, E; d\nu);\\
   -\displaystyle\frac{du}{d\nu }(t)\in N_{C(t)}  (u(t))+ f(t)  \frac{d\lambda} {d\nu}(t)  \;\;\;\nu-a.e.\,t\in
   I,
\end{cases}
$$
\end{thm}
\proof The proof is immediate by taking $A(t) = \partial \delta
_{C(t)} = N_{C(t)}$. \finproof\\

To be able to prove the main  theorem of this section, let us begin
by a compactness result.
\begin{lem}\label{Lemma 3.1}
Let for every $t\in I$, $A(t):D(A(t))\subset E\rightrightarrows E$
be a maximal monotone operator satisfying $(H_1)$, $(H_2)$ and
$(H_3)$.
 Let $X : I \rightrightarrows  E$ be a convex
weakly compact valued  measurable multi-mapping such that $X(t)
\subset M{\overline B}_E$, for all $t \in I$, where $M$ is a
nonnegative constant. Then the BVRC solutions set ${\mathcal X} :=
\{u_f : f \in S ^1_X\}$, where $S ^1_X$ denotes the set of all
$L^1(I, E; \lambda) $-selections of $X$, to the evolution inclusion
$$
\begin{cases}
u(0)=u_0\in D(A(0));\\
 u(t)\in D(A(t))\;\;\;\forall t\in I;\\
   -\displaystyle\frac{du}{d\nu }(t)\in A(t) u(t)+ f(t)   \frac{d\lambda} {d\nu}(t)  \;\;\;d\nu -a.e.\;t\in
   I,
\end{cases}
$$
 is nonempty and sequentially compact
with respect to the pointwise convergence on $I$.
\end{lem}

\proof By Theorem \ref{Theorem 3.2}, it is clear that ${\mathcal X}$
is nonempty. Now,  for each $f \in S ^1_X$, we have for almost every
$t\in I$, $\|f(t)\| \leq M$, so that, by virtue of the estimation
given in Theorem \ref{Theorem 3.2}, the solutions set ${\mathcal X}$
 is equi-BVRC. Namely
$$ u_f(t) = u_0 + \int_ { ]0, t] }  \frac{ du_f} {d\nu} (s)  d\nu(s)\;\;\;\forall  t \in I,$$
with  $\|\frac{ du_f} {d\nu} (t)\| \leq K$    $\nu $ a.e., where $K$
is a nonnegative constant which depends only on the data. Let for
each $n\in \mathbb{N}$, $f_n \in S^1_X$ and let $u_{f_n }$ be the
unique BVRC solution associated with $f_n$ to  the inclusion
 $$(P_n)
\begin{cases}
u_{f_n}(0)=u_0\in D(A(0));\\
 u_{f_n}(t)\in D(A(t))\;\;\,\forall t\in I;\\
   -\displaystyle\frac{du_{f_n}}{d\nu }(t)\in A(t) u_{f_n}(t)+ f_n(t)   \frac{d\lambda} {d\nu}(t)  \;\;\;\nu -a.e.\;t\in I.
\end{cases}
$$
By Theorem \ref{Theorem 3.1}, we know that $S^1_X $ is convex and
weakly compact in $L^1(I, E; \lambda)$, so
 we may assume that  $(f_n)$  weakly converges in  $L^1(I, E; \lambda)$ to some mapping $f  \in S^1_X$.
Since $(u_{f_n})$  is a sequence in  ${\mathcal X}$,  by weak
compactness, we may assume that $(\frac{ du_{f_n} } {d\nu})$  weakly
converges in  $L^1(I, E; \nu)$  to $z \in  L^1(I, E; \nu)$, with
$\|z(t)\|    \leq K$    $\nu$-a.e.,  and since, by $(H_3)$,  for every
$t\in I$, $( u_{f_n }(t))$ is relatively compact, we conclude by
Theorem \ref{Theorem 3.1} that $( u_{f_n })$ converges pointwise strongly to
a BVRC mapping $u$ where $u(t) = u_0+\int_{]0^,t ]} z(s) d\nu(s)$, for all
$t \in I$,   that is $ \frac{d u}{d\nu}  = z$. As $(\frac{ du_{f_n}
} {d\nu})$  weakly converges in $ L^1(I, E; \nu)$ to $\frac{d u
}{d\nu}$, $(\frac{ du_{f_n} } {d\nu}) $ Komlos converges to
$\frac{d u }{d\nu} $. Similarly $(f_n \frac{d\lambda}
{d\nu})$   Komlos converges to $ f\frac{d\lambda} {d\nu}
$. As a consequence, by applying Komlos argument to the inclusion in
$(P_n)$, taking account that $u(t)\in D(A(t))$, we get finally
$$- \frac{ du } {d\nu}(t)  - f(t)  \frac{d\lambda} {d\nu}(t) \in A(t)
u(t)\;\;\;\nu-a.e.\,t\in I$$ with $u(0)=u_0$. By uniqueness of the
solution, we conclude that $u=u_f$. Therefore ${\mathcal X}$ is
sequentially compact with respect to the pointwise convergence on
$I$. \finproof
\begin{thm} \label{Theorem 3.5}
Let for every $t\in I$, $A(t):D(A(t))\subset E\rightrightarrows E$
be a maximal monotone operator satisfying $(H_1)$, $(H_2)$ and
$(H_3)$. Let $F: I\times E\rightrightarrows  E$ be a convex  weakly compact
valued multi-mapping satisfying:
\\
(1) for each $e\in E$, the scalar function  $\delta ^*(e, F(\cdot,
\cdot))$ is  ${\mathcal B}(I)\otimes \mathcal{B}(E)$-measurable;
\\
(2) for  each $e\in E$ and for every $t\in I$, the scalar function
 $\delta ^*(e, F(t, \cdot))$ is upper semicontinuous on $E$;
\\
(3) $F(t, x)\subset  M(1+\|x\|){\overline B}_E, $ for all $(t, x)
\in I\times E$,   where $M$ is a nonnegative constant.
\\
Then the set of BVRC solutions  to the inclusion
$$(P_F)
\begin{cases}
u(0)=u_0 \in D(A(0));\\
 u(t)\in D(A(t))\;\;\;\forall t\in I;\\
\frac{du}{d\nu}\in L^{\infty}(I, E; d\nu);\\
   -\displaystyle\frac{du}{d\nu }(t)\in A(t) u(t)+ F(t, u(t))  \frac{d\lambda} {d\nu}(t)  \;\;\;\nu-a.e.\;t\in
   I,
\end{cases}
$$
is nonempty and sequentially compact with respect to the pointwise
convergence on $I$.
\end{thm}
 \proof  \textit{Step 1.} Let $u_o$ be the unique BVRC solution to the
 inclusion
 $$ (P_0)
\begin{cases}
u_o(0)=u_0 \in D(A(0));\\
 u_o(t)\in D(A(t))\;\;\;\forall t\in I;\\
   -\displaystyle\frac{du_o}{d\nu }(t)\in A(t) u_o(t) \;\;\;\nu-a.e.\;t\in
   I,
\end{cases}
$$ and let $\alpha:I\to \mathbb{R}_+$ be the unique absolutely
continuous solution to the ordinary differential equation
$$\dot \alpha(t)=M(1+\alpha(t))\;\;\forall t\in
I\;\;\textmd{with}\;\;\alpha(0)=\alpha_0:=\max_{t\in I}\|u_o(t)\|,$$
that is, $\dot \alpha(t)=M(1+\alpha_0)\exp(M t)$. Since $\dot
\alpha\in L^\infty(I, \mathbb{R}; \lambda)$, the set
 $$ S:= \big\{ h \in  L^\infty(I, E; \lambda) : \;\|h(t)\| \leq \dot\alpha(t)\;\;
 \lambda-a.e.\big\}$$ is convex $\sigma(L^\infty(I, E; \lambda), L^1(I, E;
 \lambda))$-compact, and then it is also $\sigma(L^1(I, E; \lambda), L^\infty(I, E;
 \lambda))$-compact. For any $h\in S$,
let us define
$$\Phi(h) = \Big\{ f \in  L^1(I, E; \lambda): \; f(t) \in  F(t, u_h(t))\;\;  \lambda-a.e.\,t \in I  \Big\},$$
where $ u_h$ is the unique BVRC solution to the inclusion
$$ (P_h)
\begin{cases}
u_h(0)=u_0 \in D(A(0));\\
 u_h(t)\in D(A(t))\;\;\;\forall t\in I;\\
   -\displaystyle\frac{du_h}{d\nu }(t)-h(t)\frac{d\lambda} {d\nu}(t)\in A(t) u_h(t) \;\;\;\nu-a.e.\;t\in
   I.
\end{cases}
$$
In fact, for any $h\in S$, $\Phi(h)$ is the set of $ L^1_E( I,
\mathcal B(I); \lambda)$-selections of the convex weakly compact
valued scalarly $\mathcal B(I)$-measurable mapping  $t \mapsto F(t,
u_h(t))$, by noting that  $u_h$ is BVRC, then $u_h$ is Borel, i.e.,
$(\mathcal B(I), \mathcal{B}(E))$-measurable, hence by $(1)$,   $t
\to \delta^*(e, F(t, u_h(t) ) )$ is $\mathcal B(I)$-measurable, and
then $F(\cdot,u_h(\cdot))$ admits a Borel selection. This shows the
non-emptiness of $\Phi(h)$. Using $(P_0)$ and $(P_h)$, we get by the
monotonicity of $A(t)$
\begin{equation*}\Big\langle \frac {du_h} {d\nu } (t)- \frac {du_o}
{d\nu} (t) +h(t)\frac{ d\lambda} {d\nu} (t) ,
u_h(t)-u_o(t)\Big\rangle\leq 0.\end{equation*} Since $h\in S$ it
follows that
\begin{eqnarray*}\Big\langle \frac {du_h} {d\nu } (t)-  \frac {du_o} {d\nu } (t),
u_h(t)-u_o(t)\Big\rangle \leq  \dot\alpha(t)\frac{ d\lambda} {d\nu}
(t) \|u_h(t)-u_o(t)\|.\end{eqnarray*} On the other hand, we know
that $u_h$ and $u_o$ are BVRC and have the densities $\frac {du_h}
{d\nu}$ and $\frac {du_o} {d\nu}$  relatively to $\nu$, by a result
of Moreau concerning the differential measure \cite{M3},
$\|u_h-u_o\| ^2$ is BVRC and we have
$$d\|u_h-u_o\|  ^2 \leq 2  \Big\langle u_h(\cdot) -u_o(\cdot),  \frac {du_h}  {d\nu}(\cdot)  - \frac {du_o}  {d\nu}(\cdot) \Big\rangle  d\nu,$$
so that  by integrating  on $] 0, t]$    and using the above
estimate we get \begin{eqnarray*} \frac{1}{2}\|u_h(t)-u_o(t)\|  ^2
\leq \int_{0}^t  \dot\alpha(s) \|u_h(s)-u_o(s)\| ds.\end{eqnarray*}
Thanks to Lemma \ref{lem2.8}, it follows that
\begin{eqnarray*} \|u_h(t)-u_o(t)\|
\leq \int_{0}^t  \dot\alpha(s)  ds,\end{eqnarray*} so that for all
$t\in I$, we get
\begin{eqnarray*} \|u_h(t)\|\leq \|u_o(t)\|
+ \int_{0}^t  \dot\alpha(s)  ds\leq \alpha_0 + \int_{0}^t
\dot\alpha(s) ds=\alpha(t).\end{eqnarray*} Whence, for any $h\in S$
and for all $f\in \Phi(h)$, we have by hypothesis $(3)$, for
$\lambda$-a.e $t\in I$,
\begin{eqnarray*} \|f(t)\|\leq M(1+\|u_h(t)\|)\leq M(1+\alpha(t))=\dot \alpha(t),\end{eqnarray*}
that is $\Phi(h)\subset S$, further $\Phi(h)$ is convex. Clearly, if
$h$ is a fixed point of $\Phi$ ($h\in\Phi (h)$), then  $u_h$ is a
BVRC solution of the inclusion under consideration. We show that
 $\Phi : S \rightrightarrows S$ is a convex
$\sigma(L^1(I, E; \lambda), L^\infty(I, E; \lambda))$-compact valued
 upper semicontinuous multi-mapping.
By  weak  compactness, it is enough to show that the graph of $\Phi$
is sequentially weakly  compact. Let $(h_n)\subset S$ a sequence,
which $\sigma(L^1(I, E; \lambda), L^\infty(I, E;
\lambda))$-converges to $ \bar{h}\in S$, and let $(f_n)  \subset S$
such that $f_n\in \Phi(h_n)$ and
 $(f_n)$  $\sigma(L^1(I, E; \lambda), L^\infty(I, E;
d\lambda))$-converges to $ \bar{f}\in S$. We need to show that
$\bar{f} \in \Phi(\bar{h})$. By virtue  of Lemma 3.1, we know that
the set ${\mathcal X} := \{u_h : h\in S \}$ of solutions of $(P_h)$
is sequentially compact with respect to the pointwise convergence on
$I$.
 Hence  $(u_{h_n}) $ converges pointwise  to
$u_{\bar{h}} \in \mathcal X$. Since for a.e. $t\in I$, $f_n(t) \in
F(t, u_{h_n}(t))$, the inequality
$$\big\langle 1_L(t) x,  f_n(t) \big\rangle \leq \delta^*\big(1_L(t) x,  F(t, u_{h_n}(t))\big),$$
holds for almost every  $t\in I$, for each $L\in {\mathcal L}(I)$
and for each $x \in E$. By integrating, we get
$$\int_L  \big\langle x,   f_n(t) \big\rangle dt \leq  \int_L  \delta^*\big(x,  F(t, u_{h_n}(t))\big) dt.$$
By the weak convergence of $(f_n)$ and hypothesis $(2)$, it follows
that \begin{eqnarray*}\int_L \big\langle x,   \bar{f}(t) \big\rangle
dt&=&\lim_{n\to\infty}  \int_L   \big\langle x, f_n(t) \big\rangle
dt \leq
\limsup_{n\to\infty} \int_L \delta^*\big(x,  F(t, u_{h_n}(t))\big) dt\\
&\leq&  \int_L \limsup_{n\to\infty}  \delta^*\big(x,  F(t,
u_{h_n}(t))\big) dt \leq  \int_L  \delta^*\big(x,  F(t,
u_{\bar{h}}(t))\big) dt.\end{eqnarray*} Whence we get
$$\int_L   \big\langle x,   \bar{f}(t) \big\rangle dt \leq  \int_L  \delta^*\big(x,  F(t, u_{\bar{h}}(t))\big) dt$$
for every $L \in  {\mathcal L}(I)$. Consequently $\big\langle x,
\bar{f} (t)\big\rangle \leq  \delta^*\big(x,  F(t,
u_{\bar{h}}(t))\big)$ $\lambda$-a.e. From (\cite{CV}, Prop. III.35),
we get $\bar{f}(t) \in F(t, u_{\bar{h}}(t))$ $\lambda$-a.e.
\\Applying Kakutani-Ky Fan  fixed point theorem to the convex weakly
compact valued  upper semicontinuous multi-mapping  $\Phi$ now shows
that $\Phi $ admits a fixed point, $h \in \Phi(h)$, thus proving the
existence of at least one BVRC solution  to our inclusion $(P_F)$.

\textit{Step 2.}   Compactness follows easily from the above
arguments and the pointwise   compactness   of $\mathcal X$ given in
Lemma \ref{Lemma 3.1}. \finproof
\\

Let us mention a useful result, which leads us to several
applications.
\begin{coro} \label{Corollaire 3.1}
Let for every $t\in I$, $A(t):D(A(t))\subset E\rightrightarrows E$
be a maximal monotone operator satisfying $(H_1)$, $(H_2)$ and
$(H_3)$. Let $f : I\times E  \to E$  satisfying:
\\
(i)   $f(\cdot, x) $ is $\mathcal{B}(I)$-measurable on $I$,    for
all $x \in E$.
\\
(ii)  $\|f(t, x)-f(t, y)\| \leq  M\|x-y\|$     for all  $ (t, x , y)
\in I\times E \times E$.
\\
(iii) $\|f(t, x)\|  \leq M(1+\|x\|) $ for all $(t, x) \in I \times
E$, for some nonnegative constant $M$.\\ Assume further that   there
is $ \beta \in ]0, 1[$  such that $\forall t \in I$, $0\leq  2M
\frac{ d\lambda} {d\nu} (t)d\nu  (\{t\}) \leq \beta < 1$. Then there
is a unique BVRC solution  to the problem
$$
\begin{cases}
u(0)=u_0 \in D(A(0));\\
 u(t)\in D(A(t))\;\;\;\forall t\in I;\\
\frac{du}{d\nu}\in L^{\infty}(I, E; d\nu);\\
   -\displaystyle\frac{du}{d\nu }(t)\in A(t) u(t)+ f(t, u(t))  \frac{d\lambda} {d\nu}(t)  \;\;\;\nu-a.e.\;t\in
   I.
\end{cases}
$$
\end{coro}
\proof  Existence follows from Theorem \ref{Theorem 3.5}. We need
only to prove the uniqueness. \\Suppose that there are two BVRC
solutions $u$ and $v$ to the problem under consideration, that is
$$-\frac{du}  {d\nu}  (t) - f(t, u(t))   \frac{ d\lambda} {d\nu} (t) \in   A(t) u(t),$$
$$-\frac{dv}  {d\nu}  (t) - f(t, v(t))  \frac{ d\lambda} {d\nu} (t) \in   A(t)  v(t).$$
By the monotonicity of $A(t)$ we get \begin{equation*}\Big\langle
\frac {dv} {d\nu } (t)-  \frac {du} {d\nu} (t) +\frac{ d\lambda}
{d\nu} (t) f(t,v(t))-\frac{ d\lambda} {d\nu} (t) f(t,u (t)),
v(t)-u(t)\Big\rangle\leq 0.\end{equation*} By hypothesis $(ii)$
\begin{eqnarray*}\Big\langle \frac {dv} {d\nu } (t)-  \frac {du} {d\nu } (t),
v(t)-u(t)\Big\rangle&\leq& \Big\langle \frac{ d\lambda} {d\nu} (t)
f(t,u(t))- \frac{ d\lambda} {d\nu} (t) f(t,v(t)),
v(t)-u(t)\Big\rangle\\&\leq&  M\frac{ d\lambda} {d\nu} (t)
\|v(t)-u(t)\|^2.\end{eqnarray*} On the other hand, we know that $u$
and $v$ are BVRC and have the densities $\frac {du} {d\nu}$   and
$\frac {dv} {d\nu}$  relatively to $\nu$, by a result of Moreau
concerning the differential measure \cite{M3}, $\|v-u\|  ^2$ is BVRC
and we have
$$d\|v-u\|  ^2 \leq 2  \Big\langle v(\cdot) -u(\cdot),  \frac {dv}  {d\nu}(\cdot)  - \frac {du}  {d\nu}(\cdot) \Big\rangle  d\nu,$$
so that  by integrating  on $] 0, t]$    and using the above
estimate we get \begin{eqnarray*} \|v(t)-u(t)\|  ^2 &=& \int_{]0, t]
} \frac{d\|u-v\|^2}{d\nu} d\nu(s)  \leq \int_{]0, t] } 2 \big\langle
v(s) -u(s),
\frac {dv} {d\nu}(s)  - \frac {du}  {d\nu}(s) \big\rangle d\nu(s)\\
&\leq&  \int_{]0, t] }  2M\frac{ d\lambda} {d\nu} (s)
\|v(s)-u(s)\|^2 d\nu(s).\end{eqnarray*} According to the assumption
$0\leq  2M \frac{ d\lambda} {d\nu} (t)d\nu  (\{t\}) \leq \beta < 1$
and using Lemma \ref{lem2.6}, we deduce from the last inequality
that $ u=v$ in $I$. This completes the proof. \finproof

\begin{rem} Actually Corollary \ref{Corollaire 3.1} is an  extension   of Theorem 4.1
in Adly et al (\cite{adly-hadd-thib}) dealing with the BVRC solution
for  convex  sweeping process.\end{rem}

\vskip2mm An inspection of the proofs of our preceding theorems,
allows us to state the following existence result with a fairly
general perturbation taking the form $f+F$.

\begin{thm}\label{th4.5}
Let for every $t\in I$,  $A(t): D(A(t))\subset E \rightrightarrows
E$ be a maximal monotone operator satisfying $(H_1)$,  $(H_2)$ and
$(H_3)$. Let $f : I\times E \to E$ be such that for every $x \in E$
the mapping $f(\cdot, x)$ is $\mathcal{B}(I)$-measurable  on $I$ and
such that
\\
(i)  $\|f(t, x)\| \leq m, \hskip 3pt   \forall (t,  x) \in I\times
E$, for some nonnegative constant $m$;
\\
(ii) $\|f(t, x)-f(t, y)\| \leq m \|x-y\|,  \hskip 3pt   \forall (t,
x, y) \in I\times E\times E$.
\\
Let $F:I\times E\rightrightarrows E$ be a convex   compact valued
multi-mapping  satisfying:
    \\
    $(j)$ for each $e\in E$, the scalar function
 $\delta ^*(e, F(\cdot, \cdot))$ is $\mathcal{B}(I)\otimes\mathcal{B}(E)$-measurable;
    \\
    $(jj)$  for each $e\in E$ and for every $t\in I$, the scalar function
 $\delta ^*(e, F(t, \cdot))$ is upper semicontinuous on $E$;
    \\
    $(jjj)$  $F(t,x)\subset   M (1+ \|x\|){\overline B}_E, \hskip 3pt   \forall (t,  x) \in
I\times E$, for some nonnegative constant $M$.
    \\
    Assume further that   there is  $ \beta \in ]0, 1[$  such that
    $\forall t \in I$, $0\leq  2m  \frac{ d\lambda} {d\nu} (t)d\nu  (\{t\}) \leq \beta < 1$.
Then for $u_0\in D(A(0))$,  there is a BVRC mapping $u : I
\rightarrow H$  satisfying
$$(P_{f, F})
\begin{cases}
u(0)=u_0 \in D(A(0));\\
 u(t)\in D(A(t))\;\;\;\forall t\in I;\\
\frac{du}{d\nu}\in L^{\infty}(I, E; d\nu);\\
   -\displaystyle\frac{du}{d\nu }(t)\in A(t) u(t)+ \big(f(t, u(t))+F(t, u(t))\big)  \frac{d\lambda} {d\nu}(t)  \;\;\;\nu-a.e.\;t\in
   I.
\end{cases}
$$
\end{thm}

\proof  \textit{Step 1.} We proceed   as in the proof of  Theorem
\ref{Theorem 3.5}. Let $u_o$ be the unique BVRC solution of problem
$(P_0)$.  Let $\alpha: I\to \mathbb{R}_+$ be the unique absolutely
continuous solution of the ordinary differential equation
$$\dot{\alpha}(t)=M( 1+\alpha(t)+mT)   \;\;\forall t\in I\;\;\mathrm{with}\;\;\alpha(0)=\alpha_0:=\sup_{t \in I} \|u_o(t)\|.
$$
Since $\dot \alpha   \in L^\infty(I, \mathbb{R}; \lambda) $, the set
$$
S:=\{h\in L^\infty(I, E; \lambda):\;\|h(t)\|\leq  \dot \alpha(t)
\;\mathrm{a.e.}\},
$$  is clearly convex $\sigma(L^1(I, E; \lambda),L^\infty(I, E; \lambda))$-compact.
For any $h\in S$, let us define
$$\Phi(h) = \bigg\{ \psi \in  L^1(I, E;\lambda) : \; \psi(t) \in  F(t, u_h(t))\;  \; \lambda-a.e.   \bigg\} ,$$
where $ u_h$ is the unique BVRC solution to the inclusion
$$
(P_{f,h})
\begin{cases}
u_h(0)=u_0 \in D(A(0));\\
 u_h(t)\in D(A(t))\;\;\;\forall t\in I;\\
   -\displaystyle\frac{du_h}{d\nu }(t)\in A(t) u_h(t)+ \big(f(t, u_h(t))+h(t)\big)  \frac{d\lambda} {d\nu}(t)  \;\;\;\nu-a.e.\;t\in
   I.
\end{cases}
$$
The existence and uniqueness of such a solution is granted by
Corollary \ref{Corollaire 3.1}. In fact, to see that, let us set for
any  $h \in  S$, $g(t,x)=f(t,x)+h(t)$ for all $(t,x)\in I\times E$,
then $g$ satisfies:
$$
\|g(t,x)\|\leq m+\dot \alpha(t)\leq m+\|\dot\alpha\|_\infty\;\;\;
\forall (t, x) \in I\times E,
$$
and
$$
\|g(t, x) -g(t, y)\|  \leq m \|x-y\|\;\;\;   \forall (t,  x, y) \in
I\times E\times E.$$ Furthermore, by the arguments given in the
proof of the pointwise compactness of the set $\mathcal{X}$ given in
Lemma \ref{Lemma 3.1}, it is not difficult to show that the set $
\widetilde{\mathcal{X}}=\{u_h:\;h\in S\}$, where $u_h$ is the unique
BVRC solution to problem $(P_{f,h})$, is sequentially compact with
respect
to the pointwise convergence on I.\\
Now,  using $(P_0)$, $(P_{f,h})$, the monotonicity of $A(t)$,
Moreau's inequality for BVRC functions
  and the arguments of the proof of Theorem \ref{Theorem 3.5}, for all
$t\in I$, one has the estimation
$$
\frac{1}{2}\|u_h(t)- u_o(t)\|^2\leq\int_{0}^t \|f(s, u_h(s))+
h(s)\|\|u_h(s)-u_o(s)\|ds.
$$
Thanks to Lemma \ref{lem2.8}, it follows that
$$
\|u_h(t)-u_o(t)\|\leq  \int_{0}^t\| f(s, u_h(s))+h(s)\|ds \leq
\int_{0}^t (m+\dot\alpha(s))ds,$$ so that for all $t\in I$, one gets
\begin{equation}\label{3.22}\|u_h(t)\|  \leq  \alpha_0  + \int_0 ^t  (\dot \alpha(s)  + m
)ds = \alpha(t)  +mt   \leq   \alpha(t) +mT.\end{equation}
Consequently, for any $h\in S$ and for all $\psi\in \Phi(h)$, we
have by hypothesis $(jjj)$ and \eqref{3.22},
$$\|\psi(t)\|\leq M (1 + \|u_h(t)\|)   \leq M( 1+ \alpha(t) + mT)= \dot \alpha(t)\;\;\;\lambda-a.e,$$
that is $\Phi(h) \subset  S$, and by hypothesis $(j)$, it is clear
that for any $h\in S$, $\Phi(h)$ is nonempty, further it is a convex
set.  In fact, $\Phi(h)$ is the set of $ L_E^1(I,
\mathcal{B}(I);\lambda)$-selections of the convex weakly compact
valued measurable multi-mapping  $t \mapsto F(t, u_h(t))$. Clearly,
if $h$ is a fixed point of  $\Phi$ ($h\in\Phi (h)$), then $u_h$ is a
BVRC solution to the inclusion $(P_{f, F})$. So that, we finish our
proof as in Theorem \ref{Theorem 3.5}, by using the same arguments
and the pointwise compactness of the set $\widetilde{\mathcal{X}}$.

\finproof

\section {Applications}

\subsection{Second order BVRC evolution inclusion}
\begin{thm}\label{th4.1}
Let for every $t\in I$, $A(t):D(A(t))\subset E\rightrightarrows E$
be a maximal monotone operator satisfying $(H_1)$, $(H_2)$ and
\\
$(H_3) ^{'}$ $D(A(t)) \subset X(t) \subset \gamma (t) \overline B_E$
for all $t \in I$, where $X : I \rightrightarrows E$ is a convex
compact valued Lebesgue-measurable  multi-mapping and $\gamma: I \to
\mathbb{R}$ is a nonnegative $L^1(I, \mathbb{R} ;
\lambda)$-integrable function.
\\
 Let $f : I\times
E\times E \to E$ be such that for every $x, y \in E$, the mapping
$f(\cdot, x, y)$ is $\mathcal{B}(I)$-measurable and for every $t\in
I$, the mapping $f(t, \cdot, \cdot) $ is continuous on $E\times E$
and satisfying (i)  $\|f(t, x, y)\| \leq M$    $\forall (t,  x,
y)\in I\times E\times E$,
\\
(ii) $\|f(t, z, x)-f(t, z, y)\| \leq M \|x-y\|$  $\forall (t, z, x,
y) \in I\times E\times E\times E$, for some nonnegative constant
$M$. \\Assume further that   there is  $ \beta \in ]0, 1[$  such
that $\forall t \in I$, $0\leq  2M  \frac{ d\lambda} {d\nu} (t)d\nu
(\{t\}) \leq \beta < 1$. Then for $x_0 \in E$, $u_0\in D(A(0))$
there is an absolutely continuous mapping $x: I \rightarrow  E$  and
a BVRC mapping $u : I  \rightarrow E$  with density $\frac {du}
{d\nu }$ with respect to $\nu$   satisfying
$$
\begin{cases}
x(t)=x_0+\displaystyle\int_0^ t  u(s) ds\;\;\;\forall t\in I;\\
u(0)=u_0;\\
 u(t)\in D(A(t))\;\;\;\forall t\in I;\\
\frac{du}{d\nu}\in L^{\infty}(I, E; d\nu);\\
   -\displaystyle\frac{du}{d\nu }(t)\in A(t) u(t)+ f(t, x(t),u(t))  \frac{d\lambda} {d\nu}(t)  \;\;\;\nu-a.e.\,t\in
   I.
\end{cases}
$$
\end{thm}
 \proof
Let
 $$\mathcal X := \big\{ u_\psi\in \mathcal{C}(I, E): \;  u_\psi(t) =x_0+  \int_0 ^t  \psi(s) ds,\; t \in I, \;\psi \in S ^1_X\big\}. $$
Then  $\mathcal X$ is a convex compact subset of $\mathcal{C}(I, E)$
using the compactness of the convex compact valued integral  $\int_0
^t X(s) ds$ (cf \cite{Cas2}). For any $h \in \mathcal X$, there is a
unique BVRC solution to
$$
\begin{cases}
u_h(0)=u_0;\\
 u_h(t)\in D(A(t))\;\;\;\forall t\in I;\\
   -\displaystyle\frac{du_h}{d\nu }(t)\in A(t) u_h(t)+ f(t,h(t),u_h(t))  \frac{d\lambda} {d\nu}(t)  \;\;\;\nu-a.e.\,t\in
   I.
\end{cases}
$$
with $u_h(t) = u_0 + \int_{]0, t ]}   \frac{  du_h}{d\nu}(s)
d\nu(s)$ for all $t \in I$ and      $\|\frac{  du_h}{d\nu}(t) \|
\leq  K$ $\nu$-a.e. Existence and uniqueness of such a solution is
ensured by Corollary 3.1. Indeed, for any fixed $h \in \mathcal X $,
the mapping $f_h(t, x) = f(t , h(t), x)$ satisfies $\|f_h(t, x)\|
\leq M$ for all $(t, x) \in I\times E$, \\$\|f_h(t, x) -f_h(t, y)\|
 \leq M\|x-y\| $ for all $(t,
x, y) \in I\times E\times E$,    while  the estimate of the velocity
is given  in the proof of Theorem \ref{Theorem 3.2}.
 Now for each $h \in {\mathcal X}$, let us
consider the mapping
$$\psi(h) (t) := x_0 + \int_0 ^t  u_h(s) ds\;\;\;    \forall t \in I.$$
Then it is clear that $ \psi(h) \in  {\mathcal X}$ because by $(H_3)
^{'}$, $u_h(t) \in D(A(t)) \subset X(t)$ for all $t \in I$.
 \\Our aim is to
prove that  $\psi : \mathcal X  \to \mathcal X $ is continuous in
order to obtain the existence theorem by  a fixed point approach.
This need a careful look using  the estimate of the BVRC solution
given above. It is enough  to show that, if $(h_n)$ converges
uniformly to $h$ in $\mathcal X$, then the sequence $(u_{h_n})$ of
BVRC solutions associated with $(h_n)$ of problems
 $$
\begin{cases}
u_{h_n}(0)=u_0;\\
 u_{h_n}(t)\in D(A(t))\;\;\;\forall t\in I;\\
   -\displaystyle\frac{du_{h_n} }{d\nu }(t)\in A(t) u_{h_n} (t)+ f(t,h_n(t),u_{h_n}(t))  \frac{d\lambda} {d\nu}(t) \hskip 4pt \nu-a.e.\,t\in
   I.
\end{cases}
$$
pointwise converges to  the BVRC solution $u_h$  associated with $h$
of problem
 $$
\begin{cases}
u_h(0)=u_0;\\
 u_h(t)\in D(A(t))\;\;\;\forall t\in I;\\
   -\displaystyle\frac{du_h}{d\nu }(t)\in A(t) u_h(t)+ f(t,h(t),u_h(t))  \frac{d\lambda} {d\nu}(t)  \;\;\;\nu-a.e.\,t\in
   I.
\end{cases}
$$
As  $(u_{h_n})$ is uniformly bounded and bounded in variation since
$\|u_{h_n}(t) -u_{h_n}(\tau)\| \leq K(\nu (]\tau, t])$, for $\tau
\leq t $ with   $(u_{h_n}(t)) \subset D(A(t)) \subset X(t) \subset
\gamma(t) \overline B_E$, for all $t \in I$, it  is relatively
compact, by Theorem \ref{Theorem 3.1} and  the Helly principe
\cite{Por}, we may assume that $(u_{h_n})$ pointwise converges to a
BV mapping $u(\cdot)$. Now, since for all $t\in I$, $u_{h_n }(t) =
u_0+ \int_{]0, t]} \frac{ du_{h_n }} {d\nu}(s) d\nu(s)$ and $\frac{
du_{h_n }} {d\nu} (s) \in K {\overline B}_E$ $\nu$-a.e, we may
assume that $ (\frac{ du_{h_n }} {d\nu}) $ converges weakly in $
L^1(I, E; \nu)$ to $w \in L^1(I, E; \nu)$ with $w(t)\in K {\overline
B}_E$ $\nu$-a.e,  so that
$$ \lim_{n\to\infty} u_{h_n }(t)=  u_0 + \int_{]0, t]}   w(s)  d\nu(s)\;\;\;    \forall t \in I.$$
By identifying the limits, we get
$$u(t) =  u_0 + \int_{]0, t]}   w(s)  d\nu(s) \;\;\;    \forall t \in I,$$
with  $\frac{du} {d\nu}  = w$.   Whence, using the hypothesis on
$f$, we obtain $$ \lim_{n\to\infty} f(t,  h_n (t), u_{h_n}(t)) = f(t,
h (t), u(t))\;\;\; \forall t \in I.$$  As consequence $\big(f(\cdot,
h_n(\cdot), u_{h_n} (\cdot)) \frac{d\lambda} {d\nu} (\cdot)\big)$
pointwise converges to $f(\cdot, h(\cdot), u (\cdot))
\frac{d\lambda} {d\nu} (\cdot) $. Since $ ( \frac{ du_{h_n }}
{d\nu})$ weakly converges to $\frac{du} {d\nu} $ in $ L^1(I, E;
\nu)$, we may assume that
 it Komlos
converges to  $\frac{du} {d\nu}$. For simplicity set for all $t\in
I$, $g_n(t) = f(t,  h_n (t),  u_{h_n}( t) )\frac{d\lambda} {d\nu} (t) $
and $g(t) = f(t, h (t), u_h(t))\frac{d\lambda} {d\nu} (t) $.  There is
a $\nu$-negligible set $N$ such that for $t\in I\setminus N$
$$\lim_{n\to\infty}  \frac{1} {n} \sum_{j= 1}^n  \Big(\frac{ du_{h_j }} {d\nu}(t)   +    g_j(t)\Big)  = \frac{du} {d\nu} (t)  +g(t).$$
Let  $\eta  \in D(A(t))$. From
 $$\Big\langle \frac{ du_{h_n }} {d\nu}  (t)   +g_n(t)  , u(t) - \eta  \Big\rangle = \Big\langle  \frac{ du_{h_n }} {d\nu} (t)+g_n(t)  ,  u_{h_n}(t)
 -\eta\Big\rangle
+ \Big\langle \frac{ du_{h_n }} {d\nu} (t)+g_n(t) ,  u(t)
-u_{h_n}(t) \Big\rangle,$$ let us write \begin{eqnarray*}\frac{1}{n}
\sum_{j= 1}^n \Big\langle \frac{ du_{h_j }} {d\nu} (t) +g_j (t),
u(t) - \eta \Big\rangle  &=& \frac{1}{n}  \sum_{j= 1}^n \Big\langle
\frac{ du_{h_j }} {d\nu} (t) +g_j (t) , u_{h_j} (t)-\eta \Big\rangle
\\&+& \frac{1}{n}  \sum_{j= 1}^n \Big\langle \frac{ du_{h_j }} {d\nu}
(t) +g_j(t) ,  u(t)  -u_{h_j} (t) \Big\rangle, \end{eqnarray*}  so
that
$$\frac{1}{n}  \sum_{j= 1}^n \Big\langle   \frac{ du_{h_j }} {d\nu} (t)+g_j(t),  u(t) - \eta  \big\rangle  \leq
\frac{1}{n}  \sum_{j= 1}^n \big\langle  A^0(t ,\eta) , \eta -
u_{h_j}(t) \Big\rangle + (K+M) \frac{1}{n} \sum_{j= 1}^n \|u(t)
-u_{h_j}(t))\|.$$   Passing to the limit  when $n\rightarrow
\infty$, this last inequality gives immediately
$$\Big\langle  \frac{du} {d\nu} (t)+g(t) , u(t) - \eta  \Big\rangle \leq \big\langle A^0(t,
\eta) , \eta-u(t) \Big\rangle\;\;a.e.$$ On the other hand, since for
all $t\in I$, $u_{h_n}(t)\in D(A(t))$, then $u(t)\in
\overline{D(A(t))}$.
 As a consequence, by Lemma \ref{lem2.1}, $u(t)\in D(A(t))$ and
$$-\frac{du} {d\nu} (t) \in A(t) u(t)+g(t) =  A(t) u(t)  + f(t,  h(
t), u(t))\frac{d\lambda} {d\nu} (t) \;\;\; \nu-a.e.,$$ with $u(0) =
u_0 \in D(A(0))$, so that by uniqueness $u = u_h$. Consequently, for
all $t\in I$,
$$\psi (h_n) (t) -  \psi (h) (t) =   \int_0 ^t    (u_{h_n}(s)- u_{h}(s))  ds,$$
and since  $(u _{h_n}(s)- u_{h}(s))  \rightarrow 0$  and is
pointwise  bounded;   $\|u_{h_n}(s)- u_{h}(s)\|  \leq 2\gamma(s) $,
we conclude by   Lebesgue theorem, that
$$\sup_{t \in I}  \|\psi (h_n) (t) -  \psi (h) (t) \|  \leq     \int_0^T \|u_{h_n}(s)- u_{h}(s)\|ds \longrightarrow 0,$$
 so that  $\psi (h_n)    \rightarrow  \psi (h)$  in $\mathcal{C}(I, E)$.
Whence $\psi :  \mathcal X \rightarrow  \mathcal X $ is continuous
and so has a fixed point, say $h = \psi(h) \in \mathcal X $, that
means
$$
\begin{cases}
h(t) = \psi(h) (t) = x_0+\displaystyle\int_0^t u_h (s) ds \;\;\;\forall t\in I;\\
u_h(0)=u_0;\\
 u_h(t)\in D(A(t))\;\;\;\forall t\in I;\\
   -\displaystyle\frac{du_h}{d\nu }(t)\in A(t) u_h(t)+ f(t, h(t),u_h(t))  \frac{d\lambda} {d\nu}(t)  \;\;\;\nu-a.e.\,t\in
   I.
\end{cases}
$$
\finproof

\begin{coro}\label{coro4.1}   Let  $C: I  \rightrightarrows E$ be a   convex closed  valued multi-mapping  such that
\\
$(i)$ $d_H (C(t), C(s)) \leq |r(t) - r(\tau)|$, for all  $\tau, t
\in  I $;
\\
$(ii)$ $C(t) \subset  X(t) \subset \gamma(t)\overline B_E$ for all
$t \in I$, where $X : I \rightrightarrows E$ is a convex compact
valued Lebesgue-measurable multi-mapping and $\gamma:  I
\longrightarrow \mathbb{R}$  is a nonnegative $L^1(I, \mathbb{R};
\lambda)$-integrable function.
\\
Let $f : I\times E\times E \to E$ satisfying all the hypotheses in
Theorem \ref{th4.1}.  Assume further that   there is  $ \beta \in
]0, 1[$ such that   $\forall t \in  I $, $0\leq  2M  \frac{
d\lambda} {d\nu} (t)d\nu  (\{t\}) \leq \beta < 1$. \\Then, for
$u_0\in C(0)$, $x_0 \in E$,  there is an absolutely continuous
mapping  $x: I \longrightarrow  E$ and a BVRC mapping $u : I
\longrightarrow E$ with density $\frac {du} {d\nu }$ w.r.t $\nu$
satisfying
$$
\begin{cases}
x(t) = x_0+\displaystyle\int_0^t u (s) ds\;\;\;\forall t\in I;\\
u(0)=u_0;\\
 u(t)\in C(t)\;\;\;\forall t\in I;\\
\frac{du}{d\nu}\in L^{\infty}(I, E; d\nu);\\
   -\displaystyle\frac{du}{d\nu }(t)\in N_{C(t)} (u(t))+ f(t, x(t),u(t))  \frac{d\lambda} {d\nu}(t)  \;\;\;\nu-a.e.\,t\in
   I.
\end{cases}
$$
\end{coro}

\subsection{ A new application to Skorohod problem}
We present  some  new  versions of the Skorohod problem  in Castaing
et al \cite{CMRdF} dealing with the sweeping process associated with
an absolutely continuous (or continuous) closed convex moving set
$C(t)$ in $E= \mathbb{R} ^d$. Although  we deal  with deterministic
case, it is a step forward to   Skorohod problem in  the stochastic
setting, see the recent articles by  Castaing et al \cite{CMRdF,
CMRdF2}, Rascanu \cite{Ras} and  Maticiuc et al \cite{ MRST},  for
references on this stochastic subject

\begin{thm} \label{Theorem 4.2}  Let $E = \mathbb{R} ^d$ and let for every $t\in I=[0, 1]$, $A(t):D(A(t))\subset E\rightrightarrows E$ be a maximal
monotone operator satisfying $(H_1)$ and $(H_2)$.
 Suppose that   $b: I\times
E \rightarrow E$ is a $(\mathcal B(I)\otimes \mathcal{B}(E),
\mathcal B(E))$-measurable mapping  satisfying:
\\
$(j)$  $\| b(t, x)\| \leq  M$,    for all  $(t, x ) \in I\times  E$,
for some nonnegative constant $M$.
\\
$(jj)$ For all $t\in I$, $b(t, \cdot)$ is continuous on $E$.\\ Let
$y_0\in D(A(0))$.  Then there exist a BVRC  mapping $X : I
\longrightarrow E$ and a BVRC  mapping $Y: I \longrightarrow E$
satisfying
 \[
\left\{ \begin{array}{lll}
X(0) = Y(0) = y_0; \\
X(t) = \displaystyle\int_0^t  b(s, X(s)) ds + Y(t)\;\;\;  \forall t
\in I;
\\
Y(t) \in D(A(t))\;\;\;  \forall t \in I;
\\
- \displaystyle\frac{ dY}{d\nu} (t)  \in A(t)Y(t) +
\Big(\displaystyle \displaystyle\int_0^t  b(s, X(s)) ds \Big)
\frac{d\lambda}{d\nu}(t) \;\;\; \nu-a.e.\, t\in I.
\end{array}
\right.
\]
\end{thm}
\proof     Let us set for all $t\in I$
$$X ^0(t) = y_0,\;\;
h ^1(t) =   \int_0^t b(s, X ^0(s)) ds,$$ then $h ^1$ is Lipschitz
continuous with $\|h^1(t)\|  \leq M$ for all $t\in I$. By Theorem
\ref{Theorem 3.3}, there is  a unique BVRC mapping $Y ^1 : I \to E$
solution of the problem
\[
\left\{ \begin{array}{lll}
Y^1(0) = y_0; \\
Y^1(t) \in D(A(t))\;\;\;  \forall t \in I;
\\
- \displaystyle\frac{ dY^1}{d\nu} (t)  \in A(t)Y^1(t) +h^1(t)
\displaystyle\frac{d\lambda}{d\nu}(t) \;\;\; \nu-a.e.\,t\in I
\end{array}
\right.
\]
with  $$Y ^1(t) = y_0 + \int_{]0, t] }   \frac{ dY
^1}{d\nu}(s)\;\;\; \forall t \in I$$  and $\big\|\frac{ dY
^1}{d\nu}(t) \big\| \leq K$  $\nu-a.e$,   where $K$ is a positive
constant depending on the data. Set for all $t\in I$
$$X^1(t)  =h^1(t) + Y^1(t)=    \int_0^t b\big(s, X^{0}(s)\big) ds + Y^1(t),$$
so that  $X^1$ is  BVRC. \\
 Now we
construct $X ^n$ by induction  as follows. Let for all $t\in I$
$$
h^n(t) =   \int_0^t b\big(s, X^{n-1}(s)\big)  ds.$$ Then $h^n$ is
Lipschitz continuous  with   $\|h^n(t)\| \leq M$ for all $t \in I$.
By Theorem \ref{Theorem 3.3}, there is  a unique BVRC mapping $Y ^n
: I \to E$ solution of the problem
\[
\left\{ \begin{array}{lll}
Y^n(0) = y_0; \\
Y^n(t) \in D(A(t))\;\;\;  \forall t \in I;
\\
- \displaystyle\frac{ dY^n}{d\nu} (t)  \in A(t)Y^n(t) +h^n(t)
\displaystyle\frac{d\lambda}{d\nu}(t) \;\;\; \nu-a.e.\,t\in I
\end{array}
\right.
\]
with  $$Y ^n(t) = y_0 + \int_{]0, t]}   \frac{ dY ^n}{d\nu}(s)\;\;\;
\forall t \in I$$  and $\big\|\frac{ dY ^n}{d\nu}(t) \big\| \leq K$
$\nu-a.e$.   Set for all $t\in I$
$$X^n(t)  =h^n(t) + Y^n(t)=    \int_0^t b\big(s, X^{n-1}(s)\big) ds + Y^n(t),$$
so that  $X^n$ is  BVRC, and \begin{equation}\label{4.1} - \frac{ dY
^n}{d\nu}(t) \in A(t)Y ^n(t) +  \Big( \int_0^t b\big(s,
X^{n-1}(s)\big) ds \Big)   \frac{d\lambda}{d\nu}(t)  \;\;\;
\nu-a.e.\end{equation} As $(Y^n)$ is  equi-BVRC, and for all $t\in
I$, $(Y^n(t))\subset D(A(t))$, we may assume that $(Y^n)$ converges
pointwise to a BVRC mapping $Y: I \rightarrow E$. Using the estimate
$\|\frac{ dY ^n}{d\nu}(t)\| \leq K$ $\nu$-a.e,
 we may also assume that $ (\frac{ dY ^n}{d\nu }) $
weakly  converges in $L^1(I, E; \nu)$ to $\frac{ dY }{d\nu}$, and by
hypothesis $(j)$, $ \big( b(\cdot, X^{n-1}(\cdot))\big)$  weakly
converges to $ Z\in L^1(I, E; \lambda)$. Hence   $\int_0^t b(s,
X^{n-1}(s)) ds \rightarrow \int_ 0^t Z(s) ds$  for each $t \in I$.
So we get
$$
\lim_{n \rightarrow \infty}   X^n(t)  = \lim_{n \rightarrow \infty}
\bigg(\int_0^t b(s, X^{n-1}(s)) ds + Y^n(t)\bigg) = \int_0^t Z(s) ds
+  Y(t)=:X(t). $$
 As   $(X^n(\cdot))$  pointwise converges to $X(\cdot)$   on $I$,  $ \big( b(\cdot, X^{n-1}(\cdot))\big)$  is uniformly bounded
 and, by hypothesis $(jj)$, it
 pointwise converges to $  b(\cdot, X(\cdot))$.
 Then by Lebesgue's theorem  $$\lim_{n\to\infty}\int_0^t b(s, X^{n-1}(s)) ds = \int_0^t b(s, X(s))
 ds.$$
By identifying the limits we have
$$X(t) =   \int_0^t b\big(s, X(s)\big) ds +Y(t)\;\;\; \forall t \in I.$$
 From  \eqref{4.1},  repeating   the argument
involving Komlos techniques
 we get $Y(t)\in D(A(t))$, for all $t\in I$, and
$$ - \frac{ dY}{d\nu}(t) \in A(t)Y(t) +  \int_0^t b \big(s, X(s)\big) ds\hskip 3pt   \frac{d\lambda}{d\nu}(t)  \;\;\;\nu-a.e.\,t \in I. $$
The proof is therefore complete. \finproof

\begin{thm} \label{Theorem 4.4}   Let $I:= [0, T]$ and let for every $t\in I$, $A(t):D(A(t))\subset \mathbb{R}^e\rightrightarrows  \mathbb{R}^e$
be a time dependent maximal monotone operator satisfying $(H_1)$,
with $r$ continuous instead of right continuous, and $(H_2)$.\\ Let
$z \in C^{1-var} (I, \mathbb{R}^d)$ the space of continuous
functions of bounded variation   defined on $I$ with values in
$\mathbb{R}^d$. Let ${\mathcal L } (\mathbb{R}^d,  \mathbb{R}^e)$
the space of linear mappings $f$ from $\mathbb{R}^d$ to
$\mathbb{R}^e$ endowed with the operator norm
$$ \|f\|_{\mathcal{L}} := \sup_{ x \in \overline{B}_{\mathbb{R}^d}}  \|f(x)\|_{\mathbb{R}^e}.$$
Let us consider  a class of continuous  integrand operator  $b : I
\times \mathbb{R}^e   \to  {\mathcal L } (\mathbb{R}^d,
\mathbb{R}^e)$ satisfying for a nonnegative constant $M$
\\
(a) $\|b(t, x)\|_{\mathcal{L}} \leq M$, for all $(t, x) \in I \times
\mathbb{R}^e$.
\\
(b) $\|b(t, x) -b(t, y)\|_{\mathcal{L}}  \leq   M  \|x-y\|_{
\mathbb{R}^e}$, for all $(t, x, y) \in I \times \mathbb{R}^e \times
\mathbb{R}^e$
\\
Let $a\in D(A(0))$.  Then there exist a  BVC function $X : I
\longrightarrow \mathbb{R}^e$ and a BVC function    $Y:
I\longrightarrow \mathbb{R} ^e$ satisfying
 \[
\left\{ \begin{array}{lll}
X(0) = Y(0) = a; \\
X(t) = \int_0^t  b(\tau, X(\tau)) dz_\tau + Y(t)\;\;\;  \forall t \in I; \\
- \frac{ dY}{d\nu} (t)  \in    A (t, Y (t)) +  \Big(\int_0^t b(\tau,
X(\tau)) dz_\tau\Big)  \frac{d\lambda}{d\nu}(t)\;\;\; d\nu-a.e.\,t
\in I.
\end{array}
\right.
\]
where  $\int_0^t  b(\tau, X(\tau)) dz_\tau$ denotes the
Riemann-Stieltjes integral of  the continuous function  $ b(.,
X(.))$ with respect to $z$.
\end{thm}
\proof  Let us set, for all $t\in I$,
$$X ^0(t) = a,\;\;
h ^1(t) =   \int_0^t b(\tau, a) dz_\tau$$ then by Proposition 2.2 in
  \cite{FV10}, we have
\begin{equation}\label{4.3.1} \Big\|\int_0^t b(\tau, a) dz_\tau \Big\|_{\mathcal{L}} \leq  \|b(., a) \|_{\infty}
|z|_{1-var: [0, t]}. \end{equation} Moreover
\begin{equation}\label{4.3.2}\int_0^t b(\tau, a) dz_\tau - \int_0^s b(\tau, a) dz_\tau = \int_s ^t  b(\tau, a)  dz_\tau  \end{equation}
so that  by condition $(a)$
$$\|h ^1(t) - h ^1(s) \|  \leq M   |z|_{ 1-var: [s, t] }\;\;\;\forall\, 0 \leq s \leq t \leq
T,$$ and  in particular
$$\|h^1(t)\|  \leq    M  |z|_{1-var: [0, t]}  \leq   M  |z|_{1-var: [0, T]}=: L \;\;\; \forall t \in I.$$  By
Theorem \ref{Theorem 3.3}, there is  a unique  BVC  function $Y ^1:
I \longrightarrow \mathbb{R}^e$ such that
 $$
- \frac{ dY ^1}{d\nu}(t)
 \in    A (t, Y^1 (t)) +  h^1(t)\frac{d\lambda}{d\nu}(t)\;\;\;  \nu-a.e.\,   t \in
I$$ with $\|\frac{ dY ^1}{d\nu}(t)\| \leq K$, $\nu$  a.e. where $K$
is a nonnegative constant, which depends only on the data. Set for
all $t\in I$
$$X^1(t)  =h^1(t) + Y^1(t)=    \int_0^t b\big(\tau, X^{0}(\tau)\big) dz_\tau + Y^1(t),$$
so that  $X^1$ is   BVC with
$$-  \frac{ dY ^1}{d\nu}(t)  \in  A (t, Y^1 (t))+    \Big(\int_0^t b\big(\tau, X ^0(\tau)\big)
dz_\tau\Big) \frac{d\lambda}{d\nu}(t) \;\;\;        \nu-a.e.$$ Now,
we construct $X ^n$ by induction  as follows. Let for all $t\in I$
$$
h^n(t) =   \int_0^t b\big(\tau, X^{n-1}(\tau)\big)  dz_\tau.$$ Then,
by Proposition 2.2 in \cite{FV10} and assumption $(a)$, we have the
estimate \begin{equation}\label{(*)} \|h ^n(t) - h ^n(s)\|  \leq M
|z|_{ 1-var: [s, t] }\;\;\; \forall\, 0 \leq s \leq t \leq
T\end{equation} and in particular \begin{equation}\label{(delta)}
\|h ^n(t)\| \leq M |z|_{ 1-var: [0, t] } \leq  M   |z|_{ 1-var: [0,
T] }=: L \;\;\; \forall\, 0 \leq  t \leq T.\end{equation} By Theorem
\ref{Theorem 3.3}, there is a unique BVC   function
  $Y ^n : I \longrightarrow \mathbb{R}^e $ such that
$$
- \frac{ dY ^n}{d\nu}(t)
 \in  A (t, Y^n (t))+   h^n(t)\frac{d\lambda}{d\nu}(t)\;\;\;    \nu-a.e.\,   t \in
 I$$
with  $\|\frac{ dY ^n}{d\nu}(t)\| \leq K$, $\nu$-a.e., for some
nonnegative constant $K$, which does not depend on $n$ since $L$ is
an upper bound for all $\|h^n(t)\|$. Let us set
$$X^n(t)  =h^n(t) + Y^n(t)=   \int_0^t b\big(\tau, X^{n-1}(\tau)\big) dz_\tau +
Y^n(t)\;\;\;  \forall t \in I,$$ so that $X^n$ is BVC and
\begin{equation}\label{4.3.3} - \frac{ dY ^n}{d\nu}(t)   \in A(t,Y ^n(t)) +
\Big(\int_0^t b\big(\tau, X^{n-1}(\tau)\big) dz_\tau\Big)
\frac{d\lambda}{d\nu}(t), \;\;\; \nu-a.e.\, t \in I. \end{equation}
with $\|\frac{ dY ^n}{d\nu}(t)\| \leq K$, $\nu$-a.e. We note that
$(Y^n)$ is uniformly bounded and equicontinuous, then we may assume
that    it converges uniformly to a  continuous  mapping $Y :
I\longrightarrow \mathbb{R}^e$ and  $ (\frac{ dY ^n}{d\nu}) $ weakly
converges in $L^1(I, \mathbb{R}^e; \nu)$ to $\frac{ dY }{d\nu }$
with $\|\frac{ dY }{d\nu }\| \leq K $, $\nu$-a.e. Now, from
\eqref{(*)} and \eqref{(delta)} we know that $(h ^n)$ is bounded and
equicontinuous. By Ascoli theorem, we may assume that $(h^n)$
converges uniformly to a continuous mapping $h$. Hence $X^n(t)
=h^n(t) + Y^n(t)$ converges uniformly to $X(t) := h(t)+Y(t)$. So,
$\big(b(\cdot, X^{n-1}(\cdot)) \big)$ converges uniformly to
$b(\cdot, X(\cdot))$ using the Lipschitz condition $(b)$. So that,
using Proposition 2.7 in \cite{FV10},  $\Big(\int_0^t b(\tau,
X^{n-1}(\tau)) dz_\tau\Big)$ converges uniformly to $\int_0^t
b(\tau, X(\tau)) dz_\tau$. By identifying the limits we have
$$X(t) =   \int_0^t b\big(\tau, X(\tau)\big) dz_\tau +Y(t)\;\;\;\forall t \in I.$$
As $\int_0^t b\big(\tau, X^{n-1}(\tau)\big) dz_\tau \longrightarrow
\int_0^t b\big(\tau, X(\tau)\big) dz_\tau$
 uniformly  on $I$,   $ (\frac{ dY ^n}{d\nu}) $ weakly
converges in $L^1(I, \mathbb{R}^e; \nu) $ to $\frac{ dY }{d\nu }$, $
\Big(t \mapsto\frac{ dY ^n}{d\nu}(t)+ \big(\int_0^t b (\tau,
X^{n-1}(\tau)) dz_\tau\big)   \frac{d\lambda}{d\nu}(t)\Big) $
converges weakly in $L^1(I,  \mathbb{R}^e; \nu)$ to $t \mapsto\frac{
dY }{d\nu}(t)+ \big(\int_0^t b(\tau, X(\tau) )dz_\tau\big)
\frac{d\lambda}{d\nu}(t)$, from \eqref{4.3.3} we get
$$ - \frac{ dY}{d\nu}(t)-  \Big(\int_0^t b \big(\tau, X(\tau)\big) dz_\tau \Big) \frac{d\lambda}{d\nu}(t)   \in  A (t, Y (t))
\;\;\;\nu-a.e. \, t \in I $$ by repeating the Komlos argument given
in the proofs of the above theorems. The proof is therefore
complete. \finproof

\subsection{A relaxation problem.}

In the same vein we present a new existence of BVRC  solution
dealing with a BVRC perturbation. Let   $C  : I \rightrightarrows
{\overline B}_E$ be a convex closed valued
mapping with bounded right continuous retraction, in the sense that
there is a bounded and right continuous function $\rho : I \to [0,
+\infty[$ such that $ e(C(t), C(\tau)) \leq \rho(t) -\rho(\tau)$,
for all $t, \tau \in I$ $(\tau \leq t)$. Suppose further  that
$Gr(C) \in \mathcal B(I) \otimes \mathcal B(E)$. We denote by
$$S^{BVRC}_C := \big\{ u: I  \rightarrow  E:\; u \; \textmd{is} \;BVRC,\;      u(t) \in C(t)\;\;  \forall t \in I \big\}; $$
$$S^{\infty } _C := \big\{ u \in L^\infty(I , E;  \lambda):  \;    u(t) \in C(t)  \;\;  \forall t \in I \big\}.$$
 By  Valadier  \cite{Va2},  these sets are nonempty and
  $cl (S^{BVRC}_C)=  S^{\infty}_C $,  here $cl$ denotes the closure with respect
  to the  $\sigma( L^\infty,  L^1)$-topology.  Shortly,  $S^{BVRC}_C$ is dense in
  $S^{\infty}_C $  with respect to this topology. Then we have the following.

\begin{thm}   \label{Theorem 4.3}   Let   for every $t\in I$,  $A(t): D(A(t))\subset E \rightrightarrows E$ be
a maximal monotone operator satisfying
$(H_1)$, $(H_2)$ and $(H_3)$.
 Let $a : E \to  \mathbb{R}$ be a mapping  such that
 \\
 (i) $|a(x)| \leq M$, for all $x \in E$,  for some constant $M>0$.
 \\
 (ii) $|a(x) - a(y) |  \leq M \|x-y\|$,   for all $x, y  \in E$.
\\
Assume further that   there is  $ \beta \in ]0, 1[$  such that
$\forall t \in I$, $0\leq  2M  \frac{ d\lambda} {d\nu} (t)d\nu
(\{t\}) \leq \beta < 1$. Then for $u_0\in D(A(0))$ the following
hold:
\\
(1) the  set ${\mathcal S}_{  S^{\infty }_C   }$ of BVRC solutions
to the inclusion
\[
 \left\{ \begin{array}{lll}  u(0)= u_0;\\
u(t) \in D(A(t)) \;\;\; \forall t \in I;\\
\displaystyle\frac{du} {d\nu} (t) \in L^\infty(I, E; \nu );\\
 -\displaystyle\displaystyle\frac{ du } {d\nu} (t)   \in A(t) u (t)  +  a (u(t)) h(t) \displaystyle\frac{d\lambda} {d\nu}(t) \;\;\;
 \nu-a.e.,\;
 h \in  S^{\infty }_C
 \end{array}
\right.
\]
is nonempty and   compact  with respect to the topology of pointwise convergence
on $I$.
\\
(2) The  set  ${\mathcal S}_{ S^{BVRC}_C } $  of BVRC solutions  to
the inclusion
\[
 \left\{ \begin{array}{lll}  u(0)= u_0;\\
u(t) \in D(A(t))\;\;\;  \forall t \in I;\\
\displaystyle\frac{du} {d\nu} (t) \in L^\infty(I, E; \lambda )\\
 -\displaystyle\frac{ du } {d\nu} (t)   \in A(t) u (t)  +  a(u(t)) h(t)\displaystyle\frac{d\lambda} {d\nu}(t)  \;\;\;
 \nu-a.e.,\; h \in   S^{BVRC}_C
 \end{array}
\right.
\]
is nonempty  and   is dense in  the compact set   ${\mathcal S}_{
S^{\infty }_C}$.
\end{thm}
\proof   We first note that if $h \in   S^{\infty}_C$  and   $v : I
\to E$ is BVRC, then $t \mapsto  a(v(t)) h(t) $ belongs to
$L^\infty(I , E;  \lambda)$. Let us also note that  the mapping
$f_h:I\times E\to E$ defined by $ f_h(t, x) := a(x) h(t)$ satisfies
the conditions $\|f_h(t, x)\| \leq M$ for all $(t, x)\in I\times E$,
and $ \|f_h(t, x) -f_h(t, y)\|
 \leq M \|x-y\|$ for all $(t,x, y)\in I\times E\times E$, for any $h  \in
S^{\infty}_C$. By Corollary \ref{Corollaire 3.1}, for each   $h \in
S^{\infty}_C$ (resp.  $h \in S^{BVRC}_C$), there is a unique BVRC
solution $v_h$ to the inclusion
\[
\left\{ \begin{array}{lll}
v_h(0) = u_0;  \\
v_h(t) \in  D(A(t))\;\;\; \forall t \in I;\\
\displaystyle\frac{dv_h} {d\nu} (t) \in L^\infty(I, E; \nu );\\
- \displaystyle\frac{dv_h} {d\nu} (t) \in A(t)  v_h(t)+ a(v_h(t))
h(t) \displaystyle\frac{d\lambda} {d\nu}(t)  \;\;\;   \nu-a.e.\,t\in
I.
\end{array}
\right.
\]
So the  BVRC solutions sets  are  given by:\\    ${\mathcal S}_{
S^{\infty }_C   } =   \{ v_h :\; h \in     S^{\infty }_C   \} $, and
${\mathcal S}_{ S^{BVRC}_C } =  \{ v_h :\; h \in    S^{BVRC}_C\}$.
Let $(h_n)  \subset S^{\infty }_C$. As it is shown in the proof of
Theorem \ref{Theorem 3.3}, the sequence $(v_{h_n })$ of BVRC
solutions  is equi-BVRC, namely
$$ v_{h_n}(t) =  u_0 +   \int_{]0, t] } \frac {dv_{h_n}  } {d\nu } (s)   ds
\;\;  \forall t \in I,\;\;  \big\|\frac {dv_{h_n}  } {d\nu}
(s)\big\|  \leq K\;\;\nu-a.e.$$ By weak compactness we may ensure
that   $(\frac{ dv_{h_n}} {d\nu}) \rightarrow   w$  weakly in
$L^1(I, E;  \nu)$ with $\|w(t)\| \leq K$ $\nu$-a.e., so that
$v_{h_n}(t) \to v(t)=  u_0 + \int_{]0, t] } w(s)  {d\nu } (s)   ds
$, for all $t \in I$ with  $\frac {dv} {d\nu } (s)= w(s) $. Further
it is clear that $a(v_{h_n} (t)) \rightarrow a(v(t))$. Now we prove
\\
{\bf Main fact:  $a(v_{h_n}(\cdot)) h_n (\cdot)   \rightarrow
a(v(\cdot)) h(\cdot) $ weakly in  $L^1(I, E;  \lambda)$}. Indeed,
let $g \in L^\infty(I,  E ; \lambda)$.  We have for all $t\in I$
$$\big\langle   g(t) ,  a(v_{h_n}(t)) h_n(t)  \big\rangle =   \big\langle a(v_{h_n}(t))  g(t), h_n(t) \big\rangle. $$
 It is clear that $ k_n(t):=  a(v_{h_n}(t))  g(t) $ and   $ k(t):=  a(v(t))  g(t) $
 satisfy
  $ \|k_n(t)\|  \leq M \|g(t))\|$  and  $ \|k(t)\|  \leq M \|g(t))\|$ for all $t \in I$
 with $k_n(t) \rightarrow k(t)$.
As $(h_n)$ weakly converges to  $h$ in  $ L ^1(I, E; \lambda)$ we
get
  $ \underset{n\to\infty}{\lim}\langle k_n,  h_n\rangle=   \langle k, h\rangle$
  \footnote{Here one may invoke a general fact, that on bounded subsets of $L^\infty$
 the topology of convergence
in measure coincides with  the topology  of uniform convergence  on
uniformly integrable sets, i.e. on relatively weakly compact
subsets, alias the Mackey topology. This is a  lemma due to
Grothendieck \cite{Gr} [Ch.5 \S 4 no 1 Prop. 1 and exercice]  (see
also \cite{Cas3} for a more general result concerning  the Mackey
topology  for  bounded  sequences  in $L^\infty_{E^*}$)},
  that is
$$ \lim_{n\to\infty} \int_0^T    \big\langle  g(t) , a(v_{h_n}(t)) h_n(t)    \big\rangle dt
= \int_0 ^T \big\langle g(t) ,  a(v(t)) h(t) \big\rangle dt.$$ This
shows  the required fact.
 From   $\frac{ dv_{h_n}} {d\nu}(\cdot)+ a(v_{h_n}(\cdot)) h_n(\cdot)  \frac{d\lambda} {d\nu}(\cdot)
 \rightarrow   \frac{ dv} {d\nu}(\cdot) + a(v(\cdot)) h(\cdot) \frac{d\lambda} {d\nu} (\cdot)$
 weakly   $L^1(I, E; \nu)$
and the inclusion
$$  -\frac{ dv_{h_n }} {d\nu} (t) - a (v_{h_n}(t)) h_n(t)  \frac{d\lambda} {d\nu}(t)   \in A(t)  v_{h_n}  (t)\;\;\;  \nu-a.e.,$$
by repeating the convergence limit  involving Komlos  argument given
in the proof of Theorem \ref{Theorem 3.3},    we get
$$-\frac{ dv} {dt} (t)  - a(v(t) ) h(t)  \frac{d\lambda} {d\nu}(t)  \in A(t) v(t)\;\;\;   \nu-a.e.\,t\in I.$$
By uniqueness  we have  $v= v_h$. We conclude that the mapping $\phi
: h \mapsto v_h  $ from  the compact metrizable set  $ S^{\infty }_C
\subset   L^\infty(I, E; \lambda) $ to ${\mathcal S}_{   S^{\infty
}_C }$ endowed with   the topology of pointwise convergence is
continuous. Hence $\{v_h  :\;  h \in     S^{\infty }_C\} $ endowed
with the topology of pointwise convergence is compact. Since
$S^{BVRC}_C $ is dense in $ S^{\infty }_C  $, we conclude that
$\{v_h :\; h \in S^{BVRC}_C \}$ is dense in  $\{v_h :\; h \in
S^{\infty }_C \}$. \finproof

\begin{thm}   \label{Theorem 4.4}  Let   for $t\in I$,  $A(t): D(A(t))\subset E \rightrightarrows E$ be
a maximal monotone operator satisfying $(H_1)$, $(H_2)$ and $(H_3)$.
Let
 $Ext
({\overline B}_E )$  be the set of extreme points of   ${\overline
B}_E$. Let us denote
$${\mathcal M }_  {{\overline B}_E}  := \big\{ u \in L^\infty(I , E;
\lambda):\;    u(t) \in {\overline B}_E\;\;\; \forall t \in
I\big\}$$
$${\mathcal M}_  {  Ext ({\overline B}_E )  }  := \big\{ u \in L^\infty(I , E;
\lambda):\;     u(t) \in   Ext ({\overline B}_E ) \;\;\;  \forall t
\in I \big\}.$$
 Let $a : E \to  \mathbb{R}$ be a mapping  such that     for some constant
 $M>0$,
 \\
 (i) $|a(x)| \leq M$, for all $x \in E$,
 \\
 (ii) $|a(x) - a(y) |  \leq M \|x-y\|$,   for all $x, y  \in E$.
\\
Assume further that   there is  $ \beta \in ]0, 1[$  such that
$\forall t \in I$, $0\leq  2M  \frac{ d\lambda} {d\nu} (t)d\nu
(\{t\}) \leq \beta < 1$. Then for $u_0\in D(A(0))$, the following
hold:
\\
(1) the set ${\mathcal S}_{ {\mathcal M }_ {{\overline B}_E} }$ of
BVRC solutions   of  the inclusion
\[
 \left\{ \begin{array}{lll}  u(0)= u_0;\\
u(t) \in D(A(t)) \;\;\; \forall t \in I;\\
\displaystyle\frac{du} {d\nu} (t) \in L^\infty(I, E; \nu );\\
 -\displaystyle\frac{ du } {d\nu} (t)   \in A(t) u (t)  +  a (u(t)) h(t) \displaystyle\frac{d\lambda}
 {d\nu}(t)\;\;\;  \nu-a.e., \;\;\;h \in   {\mathcal M }_  {{\overline B}_E}
 \end{array}
\right.
\]
is nonempty and   compact  for the topology of pointwise convergence
on $I$.
\\
(2) The set   ${\mathcal S}_ {  {\mathcal M}_  {  Ext ({\overline
B}_E )  } }$     of BVRC solutions   of  the inclusion
\[
 \left\{ \begin{array}{lll}  u(0)= u_0,\\
u(t) \in D(A(t)) \;\;\; \forall t \in I;\\
\displaystyle\frac{du} {d\nu} (t) \in L^\infty(I, E; \nu );\\
 -\displaystyle\frac{ du } {d\nu} (t)   \in A(t) u (t)  +  a (u(t)) h(t) \displaystyle\frac{d\lambda}
 {d\nu}(t)\;\;\;  \nu-a.e.,\;\;\; h \in    {\mathcal M }_{Ext ({\overline B}_E)}
 \end{array}
\right.
\]
is nonempty  and   is dense in  the compact set   ${\mathcal S}_{
{\mathcal M }_  {{\overline B}_E}}$.
\end{thm}
\proof  The proof is similar to the proof of  Theorem \ref{Theorem
4.3} with appropriated modifications. By  using the fact that
${\mathcal M}_ {Ext ({\overline B}_E )}$ is dense in ${\mathcal M }_
{{\overline B}_E} $, we conclude that the set $ {\mathcal
S}_{{\mathcal M}_ { Ext ({\overline B}_E )  } }$ is dense in the set
${\mathcal S}_{{\mathcal M }_  {{\overline B}_E}}$ by virtue of
Ljapunov theorem e.g \cite{Cast}.

\begin{rem}
Similar results concerning  sweeping process by convex closed moving
sets, can be easily deduced from Theorem \ref{Theorem 4.3} and
Theorem \ref{Theorem 4.4}.
\end{rem}
\subsection{Fractional evolutions}

Let $I:=[0,1]$, we   investigate in the sequel a fractional order
evolution inclusion problem ${\mathcal P} _{f, A}(D^\alpha)$ coupled
with  a time dependent maximal monotone operator  $A(t)$ with
perturbation $f$ in $E$, of the form
\begin{equation}\label{1*}
D^{\alpha }h(t)+\lambda D^{\alpha-1 }h(t) = u(t)\;\;\;   t \in
I;\end{equation}
\begin{equation}\label{2*}
I_{0^+}^{\beta }h(t)\left\vert _{t=0}\right. := \lim_{t\rightarrow
0}\int_{0}^{t}\frac{(t-s)^{\beta -1}}{\Gamma(\beta)}h(s)ds =
0;\end{equation}
\begin{equation}\label{3*} h(1)=I_{0^+}^{\gamma
}h(1)=\int\limits_{0}^{1}\frac{(1-s)^{\gamma
-1}}{\Gamma(\gamma)}h(s)ds;\end{equation}
\begin{equation*}
 - \frac{du} {dt} (t)  \in   A(t)  u(t)+ f(t, h(t), u(t))\;\;\;a.e.\,  t\in
 I,
\end{equation*}
where $\alpha \in ]1,2],\,\beta \in \lbrack 0,2-\alpha ],\, \lambda
\geq 0,\, \gamma>0$ are given constants, $D^{\alpha }$ is the
standard Riemann-Liouville fractional derivative, $\Gamma$ is the
gamma function and  $f:I \times E \times E \to E$ is a single valued
mapping.

\begin{defi}[Fractional Bochner-integral]
Let $\zeta : I \longrightarrow E$ and  $a\in I$. The fractional
Bochner-integral of order $\alpha > 0$ of the function $\zeta$ is
defined by
$$
    \big(I_{a^+}^{\alpha}\zeta\big)(t) := \int_{a}^{t}\frac{(t-s)^{\alpha -1}}{\Gamma(\alpha)}\zeta(s)ds\;\;\;\forall\, t > a.
$$
\end{defi}

We refer to \cite{KST, MR, Pod} for the general theory of Fractional Calculus and  Fractional   Differential Equations.

We denote by $W^{\alpha,1}_{B, E}(I)$ the space of  all continuous
functions in $\mathcal{C}(I, E)$ such that their Riemann-Liouville
fractional derivative of order $\alpha-1$ are continuous and their
Riemann-Liouville fractional derivative of order $\alpha$ are
Bochner-integrable.

\begin{defi}[Green function \cite{CGPT}]
 Let $\alpha \in ]1,2],\,\omega \in [0, 2-\alpha], \,k \geq 0, \,\varsigma>0$ and
 $G : I \times I \longrightarrow \mathbb{R}$ be the function defined by
\[
G(t,s) = \varphi(s)I_{0^+}^{\alpha-1}(\exp(-\lambda t)) + \left\{
\begin{aligned}
&\exp(\lambda s)I_{s^+}^{\alpha-1}(\exp(-\lambda t))\;\;\; \textmd{if}\;\;0\leq s\leq t\leq 1, \\
\\
&0\;\;\;\;\;
\;\;\;\;\;\;\;\;\;\;\;\;\;\;\;\;\;\;\;\;\;\;\;\;\;\;\;\;\;\;\;\;\;\;\;\textmd{if}\;\;0\leq
t\leq s\leq 1,
\end{aligned}
\right.
\]
where
$$\varphi(s)=\frac{\exp(\lambda s)}{\mu_0}\left[\left(I_{s^+}^{\alpha-1+\gamma}
(\exp (-\lambda t))\right)(1)-\left(I_{s^+}^{\alpha-1}(\exp
(-\lambda t)) \right)(1) \right]$$

with
$$\mu_0 = \left(I_{0^+}^{\alpha-1}(\exp (-\lambda t))\right)(1) - \left(I_{0^+}^{\alpha-1+\gamma}(\exp (-\lambda t))\right)(1).$$
\end{defi}

\begin{lem}\label{Lemma 4.1}\cite{CGPT}
    Let $G$ be the above defined function. Then  \begin{itemize}
        \item [(i)] \, $G(\cdot,\cdot)$ satisfies the following estimate
        $$
        \left|G(t, s)\right| \leq \frac{1}{\Gamma(\alpha)}\left(\frac{1+\Gamma(\gamma+1)}{|\mu_0|\Gamma(\alpha)\Gamma(\gamma+1)}+1 \right)=: M_G.
        $$
        \item [(ii)] \, If $u\in W_{B,E}^{\alpha,1}\left( I\right)$ satisfies boundary conditions \eqref{1*}, \eqref{2*} and \eqref{3*},  then
        \begin{equation*}
        u(t) =\int\limits_{0}^{1}G(t,s)\left( D^{\alpha }u\left(s\right)+\lambda D^{\alpha-1 }u(s)\right) ds \;\;\;\forall t\in I.
        \end{equation*}
        \item [(iii)] \, Let $\zeta\in L^{1}\left(I, E;\lambda\right)$ and let $u_{\zeta}:I\longrightarrow E$ be the function defined by
        \begin{equation*}
        u_{\zeta}\left( t\right): =\int\limits_{0}^{1}G(t,s) \zeta(s) ds \;\;\;\forall t\in I.
        \end{equation*}
        Then
        $$\big(I_{0^+}^{\beta}u_{\zeta}\big)(t)\left\vert _{t=0}\right. =0 \quad \text{and}\quad u_{\zeta}(1) = \left( I_{0^+}^{\gamma}u_{\zeta}\right)(1).$$
        Moreover $u_{\zeta}\in W^{\alpha,1}_{B,E}(I)$ and we have
        $$\left( D^{\alpha }u_{\zeta}\right)(t)+\lambda \left( D^{\alpha-1 }u_{\zeta}\right)(t)=\zeta(t)\;\;\;\forall t\in I.$$
\end{itemize}
        \end{lem}

From Lemma \ref{Lemma 4.1} we summarize a crucial fact.
\begin{lem}\label{lem 4.2}   {\cite{CGPT}}
Let $\zeta\in L^1(I, E; \lambda)$. Then the boundary value problem
\begin{equation*}
    \left\{
    \begin{array}{ll}
    D^{\alpha }u(t)+\lambda D^{\alpha-1 }u(t) = \zeta(t)\;\;\;\forall t \in I \\
    (I_{0^+}^{\beta}u)(t)\left\vert _{t=0}\right. = 0, \quad u(1)=(I_{0^+}^{\gamma}u)(1)
    \end{array}
    \right.
    \end{equation*}
    has a unique $W^{\alpha, 1}_{B, E}(I)$-solution defined by
    $$
    u(t)= \int_{0}^{1}G(t, s)\zeta(s)ds\;\;\;\forall  t \in I.
    $$
\end{lem}

\begin{thm} \label{Theorem 4.6} {\cite{CGPT} }  Let  $X : I  \rightrightarrows E$
be a convex  compact valued measurable multi-mapping such that $X(t)
\subset \bar{\gamma}  \overline B_E$ for all $t \in I$, where
$\bar{\gamma}$ is a nonnegative constant.   Then the $W^{\alpha,
1}_{B, E} (I)$-solutions set of problem
$$
\begin{cases}
D^{\alpha }u(t)+\lambda D^{\alpha-1 }u(t) = \zeta(t)\;\;\;  a.e.\,t\in I,\;\;\;\zeta \in S ^1_{X};\\
(I_{0^+}^{\beta}u)(t)\left\vert _{t=0}\right. = 0, \quad
u(1)=(I_{0^+}^{\gamma}u)(1)
\end{cases}
$$
is  a convex compact  subset of  $\mathcal{C}(I, E)$.
 \end{thm}

Now  we present our existence theorem for problem  ${\mathcal P}
_{f, A}(D^\alpha)$.
\begin{thm} \label{Theorem 4.7}  Let for every $t\in I$,  $A(t):   D (A(t))\subset E \rightrightarrows E$  be a maximal monotone operator satisfying
$(H_1)$, $(H_2)$ and
\\
$(H_3) ^{'}  $  $D(A(t))\subset X(t) \subset \bar{\gamma } \overline
B_E$ for all $t \in I$, where $X:I\rightrightarrows E$ is a convex
compact valued measurable multi-mapping and  $\bar{\gamma}$ is a
nonnegative constant.
\\
 Let $f : I\times
E\times E \to E$ be such that for every $x, y \in E$, the mapping
$f(\cdot, x, y)$ is $\mathcal{B}(I)$-measurable and for every $t\in
I$, the mapping $f(t, \cdot, \cdot) $ is continuous on $E\times E$
and
satisfies for some nonnegative constant $M$\\
 (i)  $\|f(t, x, y)\| \leq M$   for all $(t,  x, y)\in
I\times E\times E$,
\\
(ii) $\|f(t, z, x)-f(t, z, y)\| \leq M \|x-y\|$ for all $(t, z, x,
y) \in I\times E\times E\times E$.
\\
Assume further that   there is  $ \bar{\beta} \in ]0, 1[$  such that
$\forall t \in I$, $0\leq  2M  \frac{ d\lambda} {d\nu} (t)d\nu
(\{t\}) \leq \bar{\beta} < 1$.
\\
Then   there is   a $W^{\alpha, 1}_{B, E}(I)$ mapping   $h: I
\longrightarrow  E$ and   a BVRC   mapping $u : I  \longrightarrow
E$ satisfying  the  coupled system
$$
\begin{cases}
D^{\alpha }h(t)+\lambda D^{\alpha-1 }h(t) = u(t)  \;\;\;\forall t \in  I;\\
(I_{0^+}^{\beta}h)(t)\left\vert _{t=0}\right. = 0, \quad h(1)=(I_{0^+}^{\gamma}h)(1);\\
u(t) \in  D(A(t))   \;\;\;\forall  t \in  I;  \\
- \frac {du} {d\nu}(t)\in    A(t) u(t)+  f(t, h(t), u(t))
\frac{d\lambda} {d\nu}(t)\;\;\; \nu-a.e. \,t\in I.
\end{cases}
$$
\end{thm}
\begin{proof} Let us consider  $$\mathcal X := \big\{  u : I \to E : u(t) = \int_0 ^1 G(t, s) \zeta(s) ds,\;  \zeta\in S^1_{X}  \big\}.$$
From Theorem \ref{Theorem 4.6},   $ \mathcal X$ is  a convex compact
subset of  $\mathcal{C}(I, E)$. For any $h \in \mathcal X$, there is
a unique BVRC solution to the problem
$$
\begin{cases}
u_h(0)=u_0;\\
 u_h(t)\in D(A(t))\;\;\;\forall t\in I;\\
   -\displaystyle\frac{du_h}{d\nu }(t)\in A(t) u_h(t)+ f(t,h(t),u_h(t))  \frac{d\lambda} {d\nu}(t)  \;\;\;\nu-a.e.\,t\in
   I.
\end{cases}
$$
with $u_h(t) = u_0 + \int_{]0, t ]}   \frac{  du_h}{d\nu}(s)
d\nu(s)$ for all $t \in I$ and      $\|\frac{  du_h}{d\nu}(t) \|
\leq  K$ $\nu$-a.e. Existence and uniqueness of such a solution is
ensured by Corollary 3.1. Indeed, for any fixed $h \in \mathcal X $,
the mapping $f_h(t, x) = f(t , h(t), x)$ satisfies $\|f_h(t, x)\|
\leq M$ for all $(t, x) \in I\times E$, \\$\|f_h(t, x) -f_h(t, y)\|
\leq M\|x-y\| $ for all $(t, x, y) \in I\times E\times E$,  while
the estimate of the velocity is given  in the proof of Theorem
\ref{Theorem 3.2}.

For any $h \in \mathcal X $, we consider the mapping $\psi(h) (t) :=
\int_0^1 G(t, s) u_h(s) ds$, for all $t \in I$. Then, it is clear
that $\psi(h) \in  \mathcal X $ because  by $(H_3) ^{'}$ $u_h(t) \in
D(A(t)) \subset X(t) \subset \bar{\gamma}  \overline B_E$ for all $t
\in I$. Our aim is to prove that  $\psi : \mathcal X \to \mathcal X
$ is continuous in order  to obtain the existence theorem by  a
fixed point approach.  This need a careful look using the estimate
of the BVRC solution  given above. It is enough  to show that, if
$(h_n)$ converges uniformly to $h$ in $\mathcal X$, then the
sequence $(u_{h_n})$ of BVRC solutions associated with $(h_n)$ of
problems
 $$
\begin{cases}
u_{h_n}(0)=u_0;\\
 u_{h_n}(t)\in D(A(t))\;\;\;\forall t\in I;\\
   -\displaystyle\frac{du_{h_n} }{d\nu }(t)\in A(t) u_{h_n} (t)+ f(t,h_n(t),u_{h_n}(t))  \frac{d\lambda} {d\nu}(t) \hskip 4pt \nu-a.e.\,t\in
   I.
\end{cases}
$$
pointwise converge to  the BVRC solution $u_h$  associated with $h$
of the problem
 $$
\begin{cases}
u_h(0)=u_0;\\
 u_h(t)\in D(A(t))\;\;\;\forall t\in I;\\
   -\displaystyle\frac{du_h}{d\nu }(t)\in A(t) u_h(t)+ f(t,h(t),u_h(t))  \frac{d\lambda} {d\nu}(t)  \;\;\;\nu-a.e.\,t\in
   I.
\end{cases}
$$
This fact is ensured by repeating the machinery given in the proof
of Theorem \ref{th4.1}. We have, by using $(ii)$ of Lemma \ref{Lemma
4.1}, \begin{eqnarray*}\big\|\psi (h_n) (t) -  \psi (h) (t)\big\|
&=& \big\|\int_0 ^1 G(t, s) u _{h_n}(s) ds  - \int_0 ^1 G(t, s) u
_{h}(s)
ds\big\|\\
&= & \|\int_0 ^1  G(t, s)   \big(u_{h_n}(s)- u _{h}(s)\big)  ds\|\\
&\leq&    \int_0 ^1 M_G \|u_{h_n}(s)- u _{h}(s)\|  ds\end{eqnarray*}
 since  $(u _{h_n}(s)- u_{h}(s))  \rightarrow 0$  and is
pointwise  bounded;  $\|u_{h_n}(s)- u_{h}(s)\|  \leq 2\bar{\gamma}$,
we conclude that
$$\sup_{t \in I}  \|\psi (h_n) (t) -  \psi (h) (t) \|  \leq     \int_0^1 M_G  \|u_{h_n}(s)- u_{h}(s)\|ds \longrightarrow 0,$$
using the  Lebesgue dominated convergence theorem, so that $\psi
(h_n) - \psi (h) \rightarrow  0$  in $\mathcal{C}(I, E)$. That is,
$\psi$ has a fixed point say $h = \psi(h)  \in  \mathcal X $. It
follows that
$$
\begin{cases}
h(t) = \psi(h) (t) = \int_0^1  G(t, s) u_h (s) ds\;\;\;\forall t \in I;\\
u_{h} (t) \in  D ( A(t))\;\;\;  \forall t \in I;\\
  -\displaystyle\frac{du_h}{d\nu }(t)\in A(t) u_h(t)+ f(t, h(t),u_h(t))  \frac{d\lambda} {d\nu}(t)  \;\;\;\nu-a.e.\,t\in I.
\end{cases}
$$
Coming back to Lemma \ref{lem 4.2}  and applying  the above
notations,  this means that we have just shown that there exist a
mapping   $h \in W^{\alpha, 1}_{B, E}(I)$   and a BVRC mapping $u_h:
I \longrightarrow E$ satisfying
$$
\begin{cases}
D^{\alpha }h(t)+\lambda D^{\alpha-1 }h(t) = u_h(t)\;\;\;\forall t\in I;\\
(I_{0^+}^{\beta}h)(t)\left\vert _{t=0}\right. = 0, \quad h(1)=(I_{0^+}^{\gamma}h)(1)\\
u_{h} (t) \in  D ( A(t))\;\;\;   \forall t \in I \\
-\displaystyle\frac{du_h}{d\nu }(t)\in A(t) u_h(t)+ f(t,
h(t),u_h(t))  \frac{d\lambda} {d\nu}(t)  \;\;\;\nu-a.e.\,t\in I.
\end{cases}
$$
\end{proof}

\begin{coro}\label{coro4.2}   Let  $C: I \rightrightarrows E$ be a   convex compact  valued multi-mapping  such that
\\
$(i)$ $d_H (C(t), C(s)) \leq |r(t) - r(\tau)|$, for all  $\tau, t
\in I$;
\\
$(ii)$ $C(t) \subset  X(t) \subset \bar{\gamma} \overline B_E$ for
all $t \in I$   where $X : I \rightrightarrows E$ is a convex
compact valued measurable multi-mapping and $\gamma$ is a
nonnegative constant.
\\
Let $f : I\times E\times E \to E$ satisfying all the hypotheses in
Theorem \ref{Theorem 4.7}. Assume further that   there is  $
\bar{\beta} \in ]0, 1[$ such that $\forall t \in  [0, T] $, $0\leq
2M \frac{ d\lambda} {d\nu} (t)d\nu  (\{t\}) \leq \bar{\beta} < 1$.
Then there is a $W^{\alpha, 1}_{B, E}(I)$ mapping   $h: I
\longrightarrow E$ and a BVRC mapping $u : I  \longrightarrow E$
satisfying  the coupled system
$$
\begin{cases}
D^{\alpha }h(t)+\lambda D^{\alpha-1 }h(t) = u(t)\;\;\;\forall t \in  I;\\
(I_{0^+}^{\beta}h)(t)\left\vert_{t=0}\right. = 0, \quad h(1)=(I_{0^+}^{\gamma}h)(1);\\
u(t) \in  C(t)\;\;\;  \forall t \in  I\\
- \frac {du} {d\nu}(t)\in   N_{C(t)}(u(t))  +  f(t, h(t), u(t))
\frac{d\lambda} {d\nu}(t)\;\;\; \nu-a.e.\,t\in I.
\end{cases}
$$
\end{coro}

{\bf Comments.} Some comments are in order. An easy comparison
between Theorem \ref{th4.1} and Theorem \ref{Theorem 4.7} shows that
Theorems \ref{th4.1} deals with second order equation {\bf coupled}
with a time-dependent maximal monotone  BVRC in variation evolution,
while Theorem \ref{Theorem 4.7} deals with fractional  order
equation coupled with a time-dependent  maximal monotone   BVRC in
variation evolution, in particular, coupled  with  the sweeping
process associated with a convex compact valued BVRC mapping; see
Corollary \ref{coro4.1} and Corollary \ref{coro4.2} respectively.
Actually, there is a pioneering work in Castaing et al \cite {CGLX}
dealing with existence of AC solutions for a system of fractional
order (second order) equation coupled with a time and state
dependent absolutely continuous maximal monotone operator evolution.
A large synthesis of the study of dynamical systems coupled with
monotone set-valued operators  is given in  Brogliato et al
\cite{BT}.

\subsection{Second order BVRC  evolution inclusion with a time and state dependent maximal monotone operator}

Let $I = [0, T]$ and let $E$ be a separable Hilbert space. We begin
with a useful lemma that is inspired from \cite {CGLX}.

\begin{lem}\label{lem4.3}   Let for every $(t, x)\in I\times E$, $ A_{(t, x)} :  D ( A_{(t, x)}) \subset E \longrightarrow 2^E$
 be a maximal monotone operator satisfying:
    \\
     $(\mathcal {H}_1)$ $dis \big(A_{(t, x)},  A_{(\tau, y )}\big)  \leq   r(t) -r(\tau)  +   \|x-y\|$, for all $0 \leq \tau \leq t  \leq T$
     and for all $(x, y) \in E\times E$, where
    $r:I\longrightarrow [0, +\infty[$  is nondecreasing and right  continuous on $I$.\\
    $(\mathcal H_2)$ $\|A^0_{(t, x)} y\| \leq c( 1+\|x\|+ \|y\|)$,  for all  $(t, x, y) \in   I\times E \times D( A_{(t, x)} )$,
    for some nonnegative constant
    $c$.
\\
    Let $X: I\rightrightarrows E$ be a convex compact valued  Lebesgue-measurable multi-mapping
    such that
    $X(t) \subset \gamma (t)\overline B_E $  for all $t \in I$,
    where $\gamma $ is a nonnegative Lebesgue-integrable function and let
    $$\mathcal X := \big\{ u_\zeta \in \mathcal{C}(I, E):\;   u_f (t) =x_0+  \int_0 ^t  \zeta(s) ds,\;\forall t \in I,\; \zeta \in S ^1_X\big\}.$$
Then the following hold.
\\
$(a) $  $\mathcal X$ is a convex compact subset of $ \mathcal{C}(I,
E) $ and equi-absolutely continuous.
\\
$(b)$  For any $ h \in \mathcal X$,  the operator  $A_{(t, h(t) )}$
is equi-BVRC in variation, in the sense that there is a non
decreasing right continuous  mapping $\rho: I  \longrightarrow [0,
+\infty[$ such that
$$dis( A_{(t, h(t))}, A_{(\tau, h(\tau))}) \leq  \rho(t)-\rho(\tau)\;\;\;  \forall \, \tau, t  \in I\;(\tau\leq t).$$ Further,
$ \|A ^0_{(t, h(t) )} y \| \leq  d(1+\|y\|)$ for all $y \in D(
A_{(t, h(t) )} )$  where $d$ is a nonnegative  constant.
\end{lem}
\proof
 the proof of $(a)$   is classical using the compactness of the convex compact valued integral  $\int_0 ^t  X(s) ds$ (cf \cite{Cas2}).
 It is obvious that  for any $\zeta \in S ^1_X$, $\|u_\zeta (t) -u_\zeta(\tau)\|\leq\int_\tau ^t \gamma(s) ds$,  $0 \leq \tau \leq t \leq
 T$.
\\
$(b)$    The time-dependent maximal monotone operator  $ A_{(t,
h(t))}$ is  BVRC in  variation. Indeed, for all $0 \leq \tau \leq t
\leq T$, we have   by  $(\mathcal{H}_1)$
\begin{equation}\label{a1}
 dis (  A_{(t, h(t) )} ,    A_{(\tau, h(\tau) )} )
\leq r(t) -r(\tau)  + \|h(t)-h(\tau)\|
\leq  r(t) -r(\tau) +   \int_\tau ^t  \gamma(s)  ds  \\
= \rho (t) -\rho(\tau),
\end{equation}
where $\rho(t) = r(t)  +\int_0 ^t \gamma(s)  ds$, for all $t \in I$.
So $\rho$ is  nondecreasing right  continuous on $I$. Furthermore,
by $(\mathcal{H}_2)$ we have
\begin{equation}\label{a2}
   \|A ^0_{(t, h(t) )} y \| \leq  c( 1+ \|h(t) \|+ \|y\| )
\leq d(1+\|y\|)
\end{equation}
for all $y \in D( A_{(t, h(t) )} )$,  where $d$ is a nonnegative
generic constant, because $h$ is uniformly bounded. \finproof

\begin{thm}\label{Theorem 4.8}
 Let for every $(t, x)\in I\times E$, $ A_{(t, x)} :  D ( A_{(t, x)}) \subset E \longrightarrow 2^E$ be a maximal monotone operator
 satisfying $(\mathcal{H}_1)$, $(\mathcal{H}_2)$ and
    \\
$(\mathcal  {H}_3)$ $D(A_{(t, x)}) \subset X(t) \subset \gamma (t)
\overline B_E$ for all $(t, x) \in I\times E$, where $X : I
\rightrightarrows E$ is a convex compact valued  Lebesgue-measurable
multi-mapping and $\gamma$  is a nonnegative $L^1(I,
\mathbb{R};\lambda)$-integrable function.
\\
 Set for all $t \in I$  $\rho(t)= r(t) +\int_0^t \gamma(s)ds$  and   $\mu = \lambda +d\rho$  and
let $f : I\times E\times E \to E$ be such that for every $x, y \in
E$, the mapping $f(\cdot, x, y)$ is $\mathcal{B}(I)$-measurable and
for every $t\in I$, the mapping $f(t, \cdot, \cdot) $ is continuous
on $E\times E$ and satisfying for some nonnegative constant $M$
\\
(i)  $\|f(t, x, y)\| \leq M$   for all $(t,  x, y)\in I\times
E\times E$;
\\
(ii) $\|f(t, z, x)-f(t, z, y)\| \leq M \|x-y\|$ for all $(t, z, x,
y) \in I\times E\times E\times E$.
\\
Assume further that   there is  $ \beta \in ]0, 1[$  such that
$\forall t \in I$, $0\leq  2M  \frac{ d\lambda} {d\mu} (t)d\mu
(\{t\}) \leq \beta < 1$.
\\
Then, for any    $(x_0,u_0)\in E\times  D (  A_{(0, x_0)}) $ there
exist  an absolutely continuous mapping $x :   I \longrightarrow E$
and a BVRC  mapping $u:    I \to E$ with density $\frac {du} {d\mu
}$ with respect to $\mu$,
    such that
    \[
    \left\{ \begin{array}{lll}
    x(t) = x_0+ \int_0^t u(s)ds\;\;\; \forall t \in   I;\\
    x(0)=x_{0},\; u(0) =u_0;\\
    u(t) \in  D (A_{(t, x(t) )})\;\;\;  \forall t \in I;\\
    -  \frac{ du} {d\mu}(t)\in    A_{(t, x(t) )} u(t) + f(t, x(t), u(t))\frac{ d\lambda} {d\mu}(t)\;\;\;  \mu-a.e.\, t\in   I.\\
\end{array}
    \right.
    \]
\end{thm}
 \proof
Let
 $$\mathcal X := \big\{ u_{\zeta} \in \mathcal{C}(I, E): \;  u_{\zeta} (t) =x_0+  \int_0 ^t  \zeta(s) ds, \;\forall t \in I,\; \zeta \in S ^1_X\big\}.$$
Then by Lemma \ref{lem4.3} $(a)$ , $\mathcal X$ is a convex compact
subset of $\mathcal{C}(I, E) $ and is equi-absolutely continuous.
For each $h \in \mathcal X$, by Lemma \ref{lem4.3} $(b)$,  the
time-dependent maximal monotone operator $A_{(t, h(t) )}$ satisfies
relations \eqref{a1} and \eqref{a2}. By applying Corollary
\ref{Corollaire 3.1} with $ \nu $ replaced by $\mu$, for any $h \in
\mathcal X$, there is a unique BVRC solution of the problem
$$
\begin{cases}
u_h(0)=u_0;\\
 u_h(t)\in D(  A_{(t, h(t) )} )\;\;\;\forall t\in I;\\
   -\displaystyle\frac{du_h}{d\mu }(t)\in  A_{(t, h(t) )} u_h(t)+ f(t, h(t),u_h(t))  \frac{d\lambda} {d\mu}(t)  \;\;\;\mu-a.e.\,t\in
   I,
\end{cases}
$$
with $u_h(t) = u_0 + \int_{]0, t ]}   \frac{  du_h}{d\mu}(s) d\mu(s)$ for all $t \in I$ and      $\|\frac{  du_h}{d\mu}(t) \|  \leq  K$ $\mu$-a.e.
 Indeed, for any fixed $h \in \mathcal X $, the mapping
$f_h(t, x) = f(t , h(t), x)$ satisfies $\|f_h(t, x)\| \leq M$ for
all $(t, x) \in I\times E$, $\|f_h(t, x) -f_h(t, y)\|= \|f(t, h(t),
x) -   f(t, h(t), y)\| \leq M\|x-y\|$ for all $(t, x, y) \in I\times
E\times E$,    while  the estimate of the velocity is given the
proof of Theorem \ref{Theorem 3.2}.
 Now for each $h \in {\mathcal X}$, let us
consider the mapping
$$\psi(h) (t) := x_0 + \int_0 ^t  u_h(s) ds\;\;\;    \forall t \in I.$$
From $(\mathcal H_3)$, it is clear that $ \Phi(h) \in {\mathcal X}$.
Our aim is to prove that  $\psi : \mathcal X  \longrightarrow
\mathcal X $ is continuous in order  to obtain the existence theorem
by  a fixed point approach.   It is enough  to show that, if $(h_n)$
converges uniformly to $h$ in $\mathcal X$, then the sequence
$(u_{h_n})$ of BVRC solutions associated with $(h_n)$ of problems
 $$
\begin{cases}
u_{h_n}(0)=u_0;\\
 u_{h_n}(t)\in D(A_{(t, h_n(t) )})\;\;\;\forall t\in I;\\
   -\displaystyle\frac{du_{h_n} }{d\mu }(t)\in A_{(t, h_n(t) )} u_{h_n} (t)+ f(t, h_n(t),u_{h_n}(t))  \frac{d\lambda} {d\mu}(t) \hskip 4pt \mu-a.e.\,t\in
   I.
\end{cases}
$$
pointwise converge to  the BVRC solution $u_h$  associated with $h$
of the problem
 $$
\begin{cases}
u_h(0)=u_0;\\
 u_h(t)\in D(A_{(t, h(t)))})\;\;\;\forall t\in I;\\
   -\displaystyle\frac{du_h}{d\mu }(t)\in A_{(t, h(t) )}  u_h(t)+ f(t,h(t),u_h(t))  \frac{d\lambda} {d\mu}(t)  \;\;\;\mu-a.e.\,t\in
   I.
\end{cases}
$$
As  $(u_{h_n})$ is uniformly bounded and bounded in variation since
$\|u_{h_n}(t) -u_{h_n}(\tau)\| \leq K(\mu (]\tau, t])$, for $\tau
\leq t $ and   $u_{h_n}(t) \in D(A_{(t, h_n(t) )}) \subset X(t)
\subset \gamma(t) \overline B_E$, for all $t \in I$, and hence it is
relatively compact, by  Theorem \ref{Theorem 3.1}, we may assume
that $(u_{h_n})$ pointwise converges to a BV mapping $u(\cdot)$.
Now, since for all $t\in I$, $u_{h_n }(t) = u_0+ \int_{]0, t]}
\frac{ du_{h_n }} {d\mu} d\mu$ and $\frac{ du_{h_n }} {d\mu} (s) \in
K {\overline B}_E$ $\mu$-a.e, we may assume that $ (\frac{ du_{h_n
}} {d\mu}) $ converges weakly in $ L^1(I, E; \mu)$ to $w \in L^1(I,
E; \mu)$ with $w(t)\in K {\overline B}_E$ $\mu$-a.e,  so that
$$ \lim_{n\to\infty} u_{h_n }(t)=  u_0 + \int_{]0, t]}   w(s)  d\mu(s)\;\;\;    \forall t \in I.$$
By identifying the limits, we get
$$u(t) =  u_0 + \int_{]0, t]}   w(s)  d\mu(s) \;\;\;    \forall t \in I,$$
with  $\frac{du} {d\mu}  = w$.   Whence, using the hypothesis on
$f$, we obtain $$ \lim_{n\to\infty} f(t,  h_n (t), u_{h_n}(t)) =
f(t, h (t), u(t))\;\;\; \forall t \in I.$$  As consequence,
$\big(f(\cdot, h_n(\cdot), u_{h_n} (\cdot)) \frac{d\lambda} {d\mu}
(\cdot)\big)$ pointwise converges to $f(\cdot, h(\cdot), u (\cdot))
\frac{d\lambda} {d\mu} (\cdot) $. Since $ ( \frac{ du_{h_n }}
{d\mu})$ weakly converges to $\frac{du} {d\mu} $ in $ L^1(I, E;
\mu)$, we may assume that
 it Komlos
converges to  $\frac{du} {d\mu}$. For simplicity, set for all $t\in
I$, $z_n(t) = f(t,  h_n (t),  u_{h_n}( t) )\frac{d\lambda} {d\mu}
(t) $ and $z(t) = f(t, h (t), u_h(t))\frac{d\lambda} {d\mu} (t) $.
Hence  $(\frac{ du_{h_n }} {d\mu} + g_n(t))$ Komlos converges to
$\frac{du} {d\mu} + g(t)$. Further, we note that $u(t) \in  D  (
A_{(t, h(t) )})$ for all  $t  \in   I$. Indeed  we have $dis( A_{(t,
h_n(t)}, A_{(t, h(t) )}) \leq \|h_n(t)-h(t)\| \to 0$ and
    it is clear that  $(y_n  = A ^0_{(t, h_n(t)} u_{h_n}(t) )$ is  bounded, hence relatively weakly compact.
    By applying Lemma \ref{lem2.2} to $u_{h_n}(t) \to u(t)$ and to a convergent subsequence of  $(y_n)$
    we conclude that $u(t) \in D  ( A_{(t, h(t) )})$.  Now, apply Lemma \ref{lem2.4} to
    $A_{(t, h_n(t))}$ and  $A_{(t, h(t)} )$ to find a sequence $(\eta_n)$ such that  such that
    $\eta_n \in D( A_{(t, h_n(t) )}) , \hskip 2pt   \eta_n \to \eta,  \hskip 2pt  A ^0_{(t, h_n(t) }  \eta_n  \to    A ^0_{(t, h(t) )} u(t) $.
    From the inclusion
    \begin{equation}\label{4.8.1}-\frac{ du_{h_n }} {d\mu}(t)    \in    A_{(t, h_n(t) )} u_{h_n}(t)+
    z_n(t)\;\;\;\mu-a.e.,
    \end{equation}
    we get by the monotonicity of $A_{(t,x)}$,
    \begin{equation}\label{4.8.2}\langle  \frac{ du_{h_n }} {d\mu}(t)+z_n(t) , u_{h_n}(t) -\eta_n \rangle
    \leq  \big\langle A^0_{ (t, h_n(t))} \eta_n, \eta_n -u_{h_n}(t) \big\rangle\;\;\;\mu-a.e..
    \end{equation}
   On the other hand, since
    $$
    \big\langle \frac{ du_{h_n }} {d\mu}(t) +z_n(t) , u(t) - \eta
    \big\rangle\\
    = \big\langle  \frac{ du_{h_n }} {d\mu}(t) +z_n(t) ,  u_{h_n}(t)
    -\eta_n\big\rangle
    + \big\langle \frac{ du_{h_n }} {d\mu}(t) +z_n(t),  u(t)  -u_{h_n}(t)-(\eta-\eta_n)
    \big\rangle,
    $$
    we can write
    \begin{eqnarray*}
    &&\frac{1}{n}  \sum_{j= 1}^n \big\langle \frac{ du_{h_j }} {d\mu} (t)+z_j(t) ,  u(t) - \eta  \big\rangle
    =
    \frac{1}{n}  \sum_{j= 1}^n \big\langle  \frac{ du_{h_j }} {d\mu} (t)+z_j(t)  ,
    u_{h_j} (t)-\eta_j \big\rangle \\&+& \frac{1}{n}  \sum_{j= 1}^n
    \big\langle \frac{ du_{h_j }} {d\mu} (t) +z_j(t),  u(t)  -u_{h_j} (t)
    \big\rangle
    +\sum_{j= 1}^n
    \big\langle \frac{ du_{h_j }} {d\mu} (t) +z_j(t),  \eta_j  -\eta
    \big\rangle,
    \end{eqnarray*}
    so that
    $$\frac{1}{n}  \sum_{j= 1}^n \big\langle   \frac{ du_{h_j }} {d\mu}
    (t)+z_j(t),  u(t) - \eta  \big\rangle\\
    \leq
    \frac{1}{n}  \sum_{j= 1}^n \big\langle  A^0_{( t , h_j(t))}  \eta_j , \eta_j -
    u_{h_j}(t) \big\rangle  +(K +M) \frac{1}{n} \sum_{j=
        1}^n \|u(t)  -u_{h_j}(t))\|.
    $$
    $$
    +(K +M)   \frac{1}{n} \sum_{j=
        1}^n \|\eta_j- \eta\|.
    $$
    Passing to the limit   when
    $n\rightarrow \infty$,  this last inequality gives immediately
    $$\big\langle  \frac{du} {d\mu} (t)+z(t) , u(t) - \eta  \big\rangle \leq \big\langle A^0_{(t, h(t))}
    \eta , \eta-v(t) \big\rangle\;\;\;\mu-a.e.$$
    As a consequence,  by Lemma  \ref{lem2.1},
    we get
    $-\frac{du} {d\mu} (t) \in A_{(t, h(t) )} u(t) +z(t)$ $\mu$-a.e.
    with $u(t )\in D(A_{(t, h(t))})$ for all $t \in   I$,
    so that by uniqueness
    $u = u_h$. That is, for all
$t\in I$,
$$\psi (h_n) (t) -  \psi (h) (t) =   \int_0 ^t    (u_{h_n}(s)- u_{h}(s))  ds,$$
and since  $(u _{h_n}(s)- u_{h}(s))  \rightarrow 0$  and is
pointwise  bounded;  $\|u_{h_n}(s)- u_{h}(s)\|  \leq 2\gamma(s) $,
we conclude by   Lebesgue dominated convergence theorem, that
$$\sup_{t \in I}  \|\psi (h_n) (t) -  \psi (h) (t) \|  \leq     \int_0^T \|u_{h_n}(s)- u_{h}(s)\|ds \longrightarrow 0,$$
 so that  $\psi (h_n)  -  \psi (h) \rightarrow  0$  in $\mathcal{C}(I, E)$.
Since $\psi :  \mathcal X \rightarrow  \mathcal X $ is continuous
 it has a fixed point, say $h = \psi(h)
\in \mathcal X $, that means
$$
\begin{cases}
h(t) = \psi(h) (t) = x_0+\displaystyle\int_0^t u_h (s) ds \;\;\;\forall t\in I;\\
u_h(0)=u_0;\\
 u_h(t)\in D(A_{ (t, h(t) )} )\;\;\;\forall t\in I;\\
   -\displaystyle\frac{du_h}{d\mu }(t)\in A_{(t, h(t) )} u_h(t)+ f(t, h(t),u_h(t))  \frac{d\lambda} {d\mu}(t)  \;\;\;\mu-a.e.\,t\in
   I.
\end{cases}
$$
\finproof

As a result, we have the following corollary.

\begin{coro} Let $C: I\times E   \rightrightarrows E$ be a   convex compact  valued multi-mapping  such that
\\
$(i)$ $d_H (C(t, x), C(\tau, y)) \leq |r(t) - r(\tau)| +\|x-y\|$,
for all  $\tau, t \in  I $ and for all $(x, y) \in E\times E$.
\\
$(ii)$ $C(t, x) \subset  X(t) \subset \gamma(t)\overline B_E$ for
all $(t, x)  \in I\times E$, where $X : I \rightrightarrows E$ is a
convex compact valued  Lebesgue-measurable multi-mapping and
$\gamma: I\longrightarrow \mathbb{R}$
 is a nonnegative $L^1(I, \mathbb{R};\lambda)$-integrable function.
\\
Let for all $t\in I$, $\rho(t)= r(t) +\int_0^t \gamma(s)ds$ and $\mu
= \lambda +d\rho$  and let $f : I\times E\times E \to E$ be such
that for every $x, y \in E$, the mapping $f(\cdot, x, y)$ is
$\mathcal{B}(I)$-measurable and for every $t\in I$, the mapping
$f(t, \cdot, \cdot) $ is continuous on $E\times E$ and satisfying
for some nonnegative constant $M$
\\
(i)  $\|f(t, x, y)\| \leq M$   for all $(t,  x, y)\in
I\times E\times E$,
\\
(ii) $\|f(t, z, x)-f(t, z, y)\| \leq M \|x-y\|$ for all $(t, z, x,
y) \in I\times E\times E\times E$.
\\
Assume further that   there is  $ \beta \in ]0, 1[$  such that
$\forall t \in I$, $0\leq  2M  \frac{ d\lambda} {d\mu} (t)d\mu
(\{t\}) \leq \beta < 1$.
\\
Then, for any    $(x_0,u_0)\in E\times  C(0, x_0)$
there exist  an absolutely continuous $x :   I \to E$ and a BVRC   $u:    I \to E$ with density $\frac
{du} {d\mu }$ with respect to $\mu$,
    such that
    \[
    \left\{ \begin{array}{lll}
    x(t) = x_0+ \int_0^t u(s)ds\;\;\;\forall t \in   I;\\
    x(0)=x_{0}, u(0) =u_0;\\
    u(t) \in  C(t, x(t))\;\;\;   \forall t \in   I;\\
    -\frac{du} {d\mu} (t)  \in   N_{C(t, x(t))} u(t) + f(t, x(t), u(t))\frac{d\lambda} {d\mu} (t)\;\;\;   \mu-a.e. \,t\in   I.\\
 \end{array}
    \right.
    \]
\end{coro}

To finish the paper  we develop  a control problem where the
controls are BVRC mappings. The tools  we give has some importance
since they allow to treat some second order or some coupled system
with a time and state dependent operator $A(t, x)$.
\begin{thm}  Let for every $(t, x)\in I\times E$, $ A_{(t, x)} :  D ( A_{(t, x)}) \subset E \longrightarrow2^E$ be a maximal monotone operator satisfying
    $(\mathcal H_1)$,
    $(\mathcal H_2)$ and
    \\
    $(\mathcal H_3) ^{'} $  For any bounded set $B \subset E$,   $\bigcup_{x\in B}  D( A_{(t, x)}) $ is relatively compact.
\\
Let $\mathcal  X  := \{ u : I \to Q: \;\|u(t) -u(s)\|  \leq \rho(t)
-\rho(\tau), \;  \tau, t  \in I \,(\tau \leq t)\} $, where $Q$ is
compact subset of $E$ and $\rho : I \longrightarrow [0, +\infty[$ is
nondecreasing and right continuous. Let $\mu = \lambda+ dr + d\rho$
where $dr$  and $d\rho$ are the Stieljes measures associated with
$r$ and $\rho$. Let $f : I\times E\times E \longrightarrow E$ be
such that for every $x, y \in E$, the mapping $f(\cdot, x, y)$ is
$\mathcal{B}(I)$-measurable and for every $t\in I$, the mapping
$f(t, \cdot, \cdot) $ is continuous on $E\times E$ and satisfying
for some nonnegative constant $M$,
\\
(i)  $\|f(t, x, y)\| \leq M$   for all $(t,  x, y)\in
I\times E\times E$,
\\
(ii) $\|f(t, z ,x)-f(t, z, y)\| \leq M \|x-y\|$ for all $(t, z, x,
y) \in I\times E\times E\times E$.
\\
Assume further that   there is  $ \beta \in ]0, 1[$  such that
$\forall t \in I$, $0\leq  2M  \frac{ d\lambda} {d\mu} (t)d\mu
(\{t\}) \leq \beta < 1$.
\\
Then the following hold:
\\
$(a)$ $\mathcal  X$ is sequentially compact with respect to the pointwise convergence.
\\
$(b)$ For each $h \in \mathcal X$, the maximal monotone operator, $t
\mapsto  A_{ (t, h(t) )}$ is equi-BVRC in variation,
$$ dist (  A_{ (t, h(t) } ),    A_{ (\tau, h(\tau)} ) \leq r(t) -r(\tau) + \rho(t) -\rho(\tau)= d(r +\rho) (]\tau, t])\;(\tau\leq t).$$
$(c)$ For each $x_0\in Q$, for each $h \in \mathcal  X$ with $h(0) =
x_0$, there is a unique BVRC mapping $u_h : I \longrightarrow E$
satisfying
$$-\frac{ du_h} {d\mu} (t) \in  A_{ (t, h(t) )}u_h(t) + f(t, h(t), u_h(t))  \frac{ d\lambda} {d\mu}
(t)\;\;\; \mu-a.e.\, t \in I.
$$
$(d)$ the mapping $h \mapsto u_h$ from $\mathcal X$  to $ B
^{1-var}(I, E)$ \footnote{ $B ^{1-var}(I, E)$ denotes the space of
bounded variation $E$-valued mappings} is continuous for the
pointwise convergence, i.e, if $h_n \longrightarrow h$ then $u_{h_n}
\longrightarrow u_h $ pointwise.
\end{thm}
 \proof $(a)$ The  pointwise compactness   follows from the argument in the proof of Theorem \ref{Theorem 3.1}  using the Helly theorem
 \cite{Por}.\\
$(b)$ is obvious using assumption $(\mathcal H_1)$. \\$(c)$ Since $
A_{ (t, h(t) )}$ is equi-BVRC in variation, and  $D(  A_{ (t, h(t)
)})$ is included in a compact set for all $h \in \mathcal X$  by
assumption $(\mathcal H_3) ^{'} $, by applying  Corollary
\ref{Corollaire 3.1} with $\nu$ replace by $\mu$, there is a unique
BVRC solution $u_h$  to
$$-\frac{ du_h} {d\mu} (t) \in  A_{ (t, h(t) )}u_h(t) + f(t, h(t), u_h(t))
\frac{ d\lambda} {d\mu} (t)\;\;\; \mu-a.e. \,t \in I.$$ $(d) $
follows from the machinery given  via Komlos convergence by noting
 the estimation
 $\frac{ du_h} {d\mu} (t) \in L \overline B_H$, where $L$ is a nonnegative generic constant.
\finproof

\section{ Conclusion}
 We have established existence of
BVRC solutions for  evolution inclusions governed by time dependant
maximal monotone operators. Our results contain novelties with sharp
applications. However, there remain several issues that need full
developments,  for instance,  the existence of BVRC  solution with
different types of perturbation, e.g.  Skorohod problem and when the
perturbation is unbounded closed valued  and Lipschitz. These
considerations lead to several new research  for related problems,
for instance, the differential (equation) (variational inequalities)
(fractional inclusion) coupled with a {\bf time and state dependent
BVRC in variation}  maximal monotone operator. Actually, we are able
to solve some mentioned problems  by combining  some techniques in
Castaing et al  \cite {CGLX} with  those given here.



\end{document}